\documentclass[11pt]{amsart} 
\usepackage{amsmath}
\usepackage{enumerate}
\usepackage{tikz}
\usepackage{subcaption}
\usetikzlibrary{matrix}
\tikzset{ 
table/.style={
  matrix of nodes,
  nodes={rectangle,text width=1.75em,align=center},
  text depth=1.25ex,
  text height=2.5ex,
  nodes in empty cells
}
}
\usepackage{float}
\usepackage{amsfonts}
\usepackage{amssymb}
\usepackage{amsthm}
\usepackage{hyperref}
\hypersetup{
  pdftitle={The average cut-rank of graphs},
  pdfauthor={Huy-Tung Nguyen and Sang-il Oum}
}
\usepackage[margin=1in]{geometry}
\usepackage{thm-restate}

\newtheorem{theorem}{Theorem}[section]
\newtheorem{lemma}[theorem]{Lemma}
\newtheorem{corollary}[theorem]{Corollary}
\newtheorem{proposition}[theorem]{Proposition}

\newtheorem*{claim}{Claim}

\newenvironment{subproof}[1][\proofname]{%
  \begin{proof}[#1]%
}{%
  \end{proof}%
}

\newcommand{\dd}{\textquotedblleft}
\newcommand{\ee}{\textquotedblright}

\newcommand{\rank}{\operatorname{rank}}

\newcommand{\mac}{\mathcal}
\newcommand{\mab}{\mathbb}

\newcommand{\eq}{\equiv}
\newcommand{\vep}{\varepsilon}

\newcommand{\nin}{\notin}

\newcommand{\lr}{\left\lfloor}
\newcommand{\rr}{\right\rfloor}

\newcommand{\ero}{\mab{E}\rho}
\newcommand{\mco}{\max\rho}

\newcommand\rw{\operatorname{rw}}
\newcommand\cw{\operatorname{cw}}
\newcommand\nd{\operatorname{nd}}
\newcommand\mr{\operatorname{mr}}
\newcommand\cd{\operatorname{cd}}
\renewcommand{\subset}{\subseteq}
\newcommand\abs[1]{\lvert #1\rvert}
\begin{document}

\title{The average cut-rank of graphs}
\author{Huy-Tung Nguyen}
\address[Nguyen]{Department of Mathematical Sciences, KAIST, Daejeon, Korea.}
\email{huytung.nht@gmail.com}
\author{Sang-il Oum}
\thanks{Nguyen was supported by KAIST Undergraduate Research Participation (URP) program.
Oum was supported by the Institute for Basic Science (IBS-R029-C1).}
\address[Oum]{Discrete Mathematics Group,
  Institute for Basic Science (IBS), Daejeon, Korea.}
\address[Oum]{Department of Mathematical Sciences, KAIST, Daejeon, Korea.}
\email{sangil@ibs.re.kr}
\date\today
\begin{abstract}
	The cut-rank of a set $X$ of vertices in a graph $G$ is defined as the
	rank of the $ X \times (V(G)\setminus X)$
	matrix over the binary field whose $(i,j)$-entry is $1$ if the vertex $i$ in $X$ is adjacent to the vertex $j$ in $V(G)\setminus X$ and $0$ otherwise.
	We introduce the graph parameter called the \emph{average cut-rank}
	of a graph, defined as the expected value of the
	cut-rank of a random set of vertices.
  We show that this parameter does not increase when taking vertex-minors of graphs
  and a class of graphs has bounded average cut-rank if and only if 
  it has bounded neighborhood diversity.
  This allows us to deduce that for each real $\alpha$,
	the list of induced-subgraph-minimal graphs having average cut-rank larger than (or at least) $\alpha$ is finite.
	We further refine this by providing an upper bound on the size of obstruction and a lower bound on the number of obstructions for average cut-rank at most (or smaller than) $\alpha$ for each real $\alpha\ge0$.
	Finally, we describe explicitly all graphs of average cut-rank at most $3/2$ and determine up to $3/2$ all possible values that can be realized as the average cut-rank of some graph.
\end{abstract}
\maketitle

\section{Introduction}\label{sec:introduction}

The \emph{cut-rank function} of a graph $G$ is a function $\rho_G:2^{V(G)}\to\mab{Z}$ that maps every subset $X$ of $V(G)$ to the rank of
an $X\times (V(G)\setminus X)$ matrix over the binary field
whose rows are indexed by $X$
and columns are indexed by $V(G)\setminus X$
such that the $(i,j)$-entry is $1$ if and only if the vertex $i$ in $X$ is adjacent to the vertex $j$ in $V(G)\setminus X$.
Roughly speaking, $\rho_G(X)$ is small if the set of all edges between $X$ and $V(G)\setminus X$ form a simple structure to be described, though it could be dense.
The rank-width of a graph, introduced by Oum and Seymour~\cite{os2006}, uses 
the cut-rank function in its definition.

One of the most important properties of the cut-rank function is that it is preserved under the operation called local complementation. 
The \emph{local complementation} at a vertex $v$ of a graph~$G$
is an operation to obtain a new graph $G*v$ from $G$ by complementing in the neighborhood of $v$. In other words, for all pairs $x$, $y$ of neighbors of $v$, we delete $xy$ if $x$, $y$ are adjacent and add an edge $xy$ otherwise to obtain $G*v$.
A graph $H$ is a \emph{vertex-minor} of a graph $G$ if $H$ is an induced subgraph of a graph that can be obtained from $G$ by some sequence of local complementations.
Since local complementation preserves the cut-rank function~\cite{oum2005},  if $H$ is a vertex-minor of $G$ and $X\subseteq V(H)$, then 
$\rho_H(X)\le \rho_G(X)$. It follows that the class of graphs of rank-width at most $k$ is closed under taking vertex-minors~\cite{oum2005}.

It turns out that some of the theory developed for graph minors by Robertson and Seymour (see a survey of Lov\'asz~\cite{Lovasz2006}) can be generalized for vertex-minors. For instance,  Oum~\cite{oum2008} showed that graphs of bounded rank-width are well-quasi-ordered under the vertex-minor relation and conjectured that graphs are well-quasi-ordered under the vertex-minor relation. If true, every class of graphs closed under taking vertex-minors would be characterized by a finite list of forbidden vertex-minors.
This is why for each $k$, the class of graphs of rank-width at most $k$ is characterized by finitely many forbidden vertex-minors~\cite{oum2005}.
Now there are many interesting problems regarding vertex-minors of graphs and yet we have only a few graph parameters that do not increase by taking vertex-minors. We need more examples to develop the theory of graph structure with respect to vertex-minors.

We aim to introduce one such graph parameter, called the average cut-rank. 
The \emph{average cut-rank} of a graph $G$, denoted by $\ero(G)$, is the expectation of $\rho_G(X)$
for a uniformly chosen random subset $X$ of $V(G)$.
We will show that 
if $H$ is a vertex-minor of $G$, then $\ero(H)\le \ero(G)$.

Initially rank-width was introduced to study clique-width of a graph, introduced by Courcelle and Olariu~\cite{co2000}. Oum and Seymour~\cite{os2006} showed that 
\[\rw(G)\le \cw(G)\le 2^{\rw(G)+1}-1,\] 
where $\rw(G)$,
$\cw(G)$ denotes the rank-width, the clique-width of $G$, respectively.
For each $k$, the class of graphs of clique-width at most $k$ is closed under taking induced subgraphs but not under vertex-minors.
Thus, rank-width is `tied' to clique-width and yet it behaves better with vertex-minors than clique-width.

The average cut-rank also has such tied parameters. First let us describe a few graph parameters.
\begin{itemize}
  \item The \emph{neighborhood diversity} of a graph $G$, denoted by $\nd(G)$, is the number of  equivalence classes on $V(G)$ where two vertices $x$, $y$ are equivalent if and only if $x=y$ or $x$, $y$ are twins in $G$. This was introduced by  Lampis~\cite{lampis2012} in 2012
  and an equivalent concept appeared earlier in Ding and Kotlov~\cite{DK2006}.
  \item The \emph{maximum cut-rank} of a graph $G$ is  $\mco(G):=\max_{S\subseteq V(G)}\rho_G(S)$.
  \item For a field $\mathbb F$, the \emph{minimum rank} of an $n$-vertex graph $G$, denoted by $\mr(\mathbb F,G)$ is the minimum rank of an $n\times n$ symmetric matrix $A=(a_{ij})$ over $\mathbb F$ such that for all $i\neq j$, $a_{ij}\neq 0$ if and only if the $i$-th vertex is adjacent to the $j$-th vertex. Note that any element of $\mathbb F$ is allowed in the diagonal entry of $A$. For more about the minimum rank of a graph, readers are referred to a survey by Fallat and Hogben~\cite{FH2007}. We write $\mathbb F_q$ to denote the finite field with $q$ elements.
  \item The \emph{clique delta-cover number} of a graph $G$, denoted by $\cd(G)$, is the minimum integer $t$ such that there exist $t$ complete graphs $G_1$, $G_2$, $\ldots$, $G_t$ 
  with the property that $E(G)=E(G_1)\triangle E(G_2)\triangle \cdots \triangle E(G_t)$,
  where $\triangle$ denotes the symmetric difference operation.
\end{itemize}
We prove that all these parameters are tied to each other, when $\mathbb F$ is a finite field as follows. 
\begin{restatable}{theorem}{thmnd}\label{thm:nd}
  Let $G$ be a graph with at least one edge. Then
  \begin{enumerate}[\rm (i)]
    \item $\ero(G)< \mco(G)\le \mr(\mathbb F_2,G)\le \nd(G)<2^{2\mco(G)+2}\le 2^{8\ero(G)+2}$, 
    \item $\ero(G)<\cd(G)\le\frac{3}{2} \mr(\mathbb F_2,G)\le\frac{3}{2}\nd(G)\le \frac{3}{2} 2^{\cd(G)}$, and 
    \item $\nd(G)\le \abs{\mathbb F}^{\mr(\mathbb F,G)}\le \abs{\mathbb F}^{\nd(G)}$ for every finite field $\mathbb F$.
  \end{enumerate}
\end{restatable}
Ding and Kotlov~\cite[Lemma 2.3]{DK2006} showed that 
graphs of bounded neighborhood diversity are well-quasi-ordered under the induced subgraph relation.
Independently, Ganian, Hlin\v{e}n\'y, Ne\v{s}et\v{r}il, Obdr\v{z}\'{a}lek, and Ossona de Mendez~\cite{GHNOO2019}
showed that a class of graphs has bounded neighborhood diversity if and only if it has \emph{shrub-depth} $1$
and proved that every class of graphs of bounded shrub-depth
is well-quasi-ordered under the induced subgraph relation. 
Therefore we deduce the following corollary.
\begin{corollary}\label{cor:wqo}
  Every class of graphs of bounded average cut-rank
  is well-quasi-ordered under the induced subgraph relation.
\end{corollary}

Note that for a vertex $v$ of $G$, $\ero(G- v)\le \ero(G)\le \ero(G-v)+1$.
Together with this easy fact, Corollary~\ref{cor:wqo} implies that
for each real $\alpha$,
there are only finitely many induced-subgraph-minimal graphs of average cut-rank at least $\alpha$ up to isomorphism,
because those graphs have average cut-rank at most $\alpha+1$.

We not only prove that there are finitely many of those graphs, but also provide an explicit upper bound on the number of vertices in each of them.
Let us write $\log$ to denote $\log_2$, the binary logarithm. For every real $x$, let $\lfloor x\rfloor$ be the greatest integer not exceeding $x$ and $\{x\}:=x-\lfloor x\rfloor$ be the fractional part of $x$. For $\vep\in[0,1)$, we define a sequence $\{x_n(\vep)\}_{n\ge0}$ by
\begin{align*} x_0(\vep)&=\max(\lfloor2-\log(1-\vep)\rfloor,5),\\
  x_{n}(\vep)&=2^{8n+10}\lfloor x_{n-1}(\vep)-\log (1-\{2^{x_{n-1}(\vep)}\vep/2\})+1\rfloor.
\end{align*}
It is not hard to see that $x_n(\vep)\ge 2^{\Omega(n^2)}$
where the constant factor in the exponential term depends on~$\vep$.
Now we are ready to present our second theorem.
\begin{restatable}{theorem}{sizeobstruction}\label{thm:main2}
  Let $\alpha\ge 0$ and $G$ be a graph with no isolated vertices.
  If $\ero(G)\ge \alpha$
  and $\ero(G-v)\le \alpha$ for all vertices $v$ of $G$, 
  then $\abs{G}< x_{\lfloor \alpha\rfloor}(\{\alpha\})$.
\end{restatable}

Theorem~\ref{thm:main2} implies that
induced-subgraph-minimal graphs of average cut-rank at least $\alpha$
have bounded number of vertices for each $\alpha$.
Our third theorem shows that  the number of such graphs cannot be too small.
Indeed we prove a stronger statement in terms of vertex-minors.
Our third theorem says that
if we have a set $\mathcal S$ of graphs characterizing average
cut-rank at most $\alpha$ in terms of forbidding graphs in $\mathcal
S$ as a vertex-minor,
then $\abs{\mathcal S}$ cannot be too small.
We remark that $\mathcal S$ does not need to contain all vertex-minor-minimal graphs having average cut-rank more than $\alpha$, because if two graphs are locally equivalent (which we define in Section~\ref{sec:basic}), then $\mathcal S$ does not need to have both of them.

\begin{restatable}{theorem}{boundforforbiddenvertexminor}\label{thm:main4}
  There is some universal constant $c>0$ so that the following holds.
  For every $\vep\in [0,1)$ and $n\ge0$,
  let $\mathcal S$ be a set of graphs such that
  the average cut-rank of a graph $G$ is at most (or less than) $\vep+n$
  if and only if no graph in $\mathcal S$ is isomorphic to a vertex-minor of $G$.
  Then $\mathcal S$ contains at least $2^{cn\log (n+1)}$ graphs.

\end{restatable}

Our final theorem characterizes graphs of average cut-rank at most $3/2$ completely
and determines all possible reals up to $3/2$ that can be realized as the average cut-rank
of some graph.
For two graphs $G$ and $H$, let $G+H$ be the \emph{disjoint union} of $G$ and $H$, and for an integer $m$, let $mG$ be the disjoint union of $m$ copies of $G$. For every $k\ge 0$, let $E_k$ be $K_{1,k+1}$ with one edge subdivided. 
\begin{restatable}{theorem}{smallgraphs}\label{thm:3/2}
  Let $G$ be a graph with no isolated vertices.
  Then $G$ has average cut-rank at most $3/2$
  if and only if it is 
  isomorphic to a vertex-minor of one of $P_5$, $3K_2$, $2P_3$, $K_{1,k+1}$, $K_2+K_{1,k+1}$, and $E_k$ for $k\ge 0$.
  In addition, the set of all possible values for average cut-rank of graphs in the interval $[0,3/2]$ is
\[ \left\{1-\frac{1}{2^k}:k\ge0\right\}\cup\left\{\frac{3}{2}-\frac{1}{2^{k+1}}:k\ge0\right\}\cup\left\{\frac{3}{2}-\frac{3}{2^{k+2}}:k\ge0\right\}\cup\left\{\frac{3}{2}\right\}. \]
\end{restatable}

This paper is organized as follows.
In Section~\ref{sec:prelim} we recall basic definitions and results.
In Section~\ref{sec:equiv} we discuss an equivalence relation involving cut-rank functions.
We introduce and prove basic tools on the average cut-rank in Section~\ref{sec:basic}.
Sections~\ref{sec:wqo},~\ref{sec:upperbound},~\ref{sec:forbidden},~and~\ref{sec:3/2} present the proofs of Theorems~\ref{thm:nd},~\ref{thm:main2},~\ref{thm:main4},~and~\ref{thm:3/2} respectively.

\section{Preliminaries}\label{sec:prelim}
\subsection{Basic notions on graphs}
For all positive integers $k$,
let $P_k$ be the path on $k$ vertices, $C_k$ be the cycle on $k$ vertices,
$K_k$ be the complete graph on $k$ vertices, and
$K_{m,k}$ be the complete bipartite graph on $m$ vertices one side and $k$ vertices the other side. For the star $K_{1,k}$, we call the vertex at the singleton side the \emph{central vertex}. (If $k=1$ then we fix one vertex to be called central.)

For a graph $G$, denote $V(G),E(G),A(G)$, respectively, for its vertex set, edge set, and adjacency matrix. For disjoint sets $S,T\subseteq V(G)$, let $N_G(S,T)$ be the set of vertices in
 $T$ adjacent to at least one member in $S$.
For $v\in V(G)$, let $N_G(v,S):=N_G(\{v\},S)$ and let $N_G(v):=N_G(v,V(G))$ be the set of all  \emph{neighbors} of $v$ in $G$. Let 
 $d_G(v):=|N_G(v)|$ be the \emph{degree} of $v$ in $G$.
A vertex is \emph{isolated} if it has degree zero, and a \emph{leaf} if it has degree one.

Let $G[S]$ be the subgraph of $G$ induced on the vertex set $S$; in this case we say $G[S]$ is an \emph{induced subgraph} of $G$, and set $G-S:=G[V(G)\setminus S]$ as well as $G-v:=G-\{v\}$.
For any two disjoint subsets $X,Y$ of $G$, denote by $G[X,Y]$ the induced bipartite subgraph of $G$ with bipartition $(X,Y)$ consisting of edges having one end in $X$ and the other in $Y$.
For simplicity, set $|G|:=|V(G)|$, and we sometimes write $A_G$ instead of $A(G)$.

Let the \emph{complement} of $G$, denoted by $\overline{G}$, be the graph with vertex set $V(G)$ and edge set $\{uv:u\ne v,uv\nin E(G)\}$.

Two distinct vertices $x$, $y$ of $G$ are called \emph{twins} if
$N_G(x)\setminus\{x,y\}=N_G(y)\setminus\{x,y\}$.
If, in addition, they are adjacent then we call them \emph{true twins}, otherwise we call them \emph{false twins}.

In $G$, a subset $S\subseteq V(G)$ is a \emph{clique} if every two vertices in $S$ are adjacent, and an \emph{independent set} every two vertices in $S$ are nonadjacent.

For two disjoint subsets $A,B\subseteq V(G)$,
$A$ is \emph{complete} to $B$ if every vertex in $A$ is adjacent to all vertices of $B$
and
$A$ is \emph{anticomplete} to $B$ if every vertex in $A$ is nonadjacent to all vertices of $B$.

For two sets $A$ and $B$, let
$A\triangle B:=(A\setminus B)\cup(B\setminus A)$.
For two graphs $G_1$ and $G_2$,
let the \emph{symmetric difference} of $G_1$ and $G_2$,
denoted by $G_1\triangle G_2$, be the graph with vertex set $V(G_1)\cup V(G_2)$
and edge set $E(G_1)\triangle E(G_2)$.
When $E_1\cap E_2=\emptyset$ and $G=G_1\triangle G_2$
we say $G$ admits a \emph{decomposition} into $G_1$ and $G_2$.

For a subset $S$ of $V(G)$, \emph{identifying} $S$ is the operation of replacing all vertices in $S$ by a new vertex and joining it to every vertex in $N_G(S,V(G)\setminus S)$. For an equivalence relation $\equiv$ on $V(G)$, the \emph{quotient graph} of $G$ induced by $\equiv$ is the graph obtained from $G$ by identifying each equivalence class $C$ of $(V(G),\equiv)$ to a vertex denoted by $C$.

For two graphs $G_1$ and $G_2$, we say $G_1$ is \emph{isomorphic} to $G_2$,
if there is a bijection $\varphi:V(G_1)\to V(G_2)$ satisfying
for $u,v\in V(G)$, 
$\varphi(u)\varphi(v)$ is an edge of $G_2$ if and only if $uv$ is an edge of~$G_1$. 
\subsection{Local complementations and vertex-minors}
For a graph $G$ and its vertex $v$, let $G*v$ be the graph obtained from $G$ by switching all adjacencies between neighbors of~$v$. To be precise, two vertices $x$ and $y$ are adjacent in $G*v$ if and only if in $G$, either
\begin{enumerate}
\item they are adjacent and at least one of them is non-adjacent to $v$, or
\item they are nonadjacent but both are adjacent to $v$. 
\end{enumerate}
Indeed, $V(G*v)=V(G)$ and $G*v*v$ is $G$ itself for every $v\in V(G)$. We call such an operation the \emph{local complementation} at $v$. We say two graphs are \emph{locally equivalent} if one can be obtained from the other by a series of local complementations. 

We say that a graph $H$ is a \emph{vertex-minor} of $G$ if it can be obtained from $G$ by a series of local complementations and vertex deletions. A simple observation points out that given such a series, we may rearrange the operations so that all the local complementations are executed before the vertex deletions without changing the output graph. Thus, if $H$ is a vertex-minor of $G$, then $H$ is actually an induced subgraph of a graph locally equivalent to $G$.

For an edge $uv$ of $G$, the \emph{pivot} of $G$ on $uv$ is
an operation to obtain a graph, denoted by $G\wedge uv$, from $G$ by three local complementations, $G*u*v*u$.
This is well defined because $G*u*v*u=G*v*u*v$ whenever $u$, $v$ are adjacent, see~\cite[Proposition 2.1]{oum2005}.
\subsection{Cut-rank}
For a matrix $M:=(m_{ij}:i\in R, j\in C)$, let $\rank(M)$ be its rank. If $X\subseteq R$ and $Y\subseteq C$, denote by $M[X,Y]$ the submatrix of $M$ obtained by taking the rows indexed by $X$ and the columns indexed by $Y$ so that $M[X,Y]=(m_{ij}:i\in X,j\in Y)$.

For a graph $G$ and two disjoint subsets $X$ and $Y$,
let us write $\rho_G^*(X,Y)=\rank (A_G[X,Y])$
where $A_G$ is considered as a matrix over the binary field.
The \emph{cut-rank function} of a graph $G$
is a function $\rho_G:2^{V(G)}\to\mathbb Z$ such that  
$\rho_G(S):=\rho_G^*(S,V(G)\setminus S)$.
This implies immediately that $\rho_G$ is symmetric, that is, $\rho_G(S)=\rho_G(V(G)\setminus S)$ for all $S\subseteq V(G)$.

In this paper we need the following property of cut-rank functions, which shows that local complementations preserve the cut-rank function of a graph $G$.
\begin{proposition}[{{Oum~\cite[Proposition 2.6]{oum2005}}}]\label{prop:cutrank}For a graph $G$
and $v\in V(G)$,
we have $\rho_G(S)=\rho_{G*v}(S)$ for all $S\subseteq V(G)$.
\end{proposition}
\subsection{Well-quasi-ordering and forbidden lists}
Given a set $\mac{X}$ and a relation $\le$ on $\mac{X}$,
$(\mac{X},\le)$ is a \emph{quasi-order} if
\begin{enumerate}[(i)]
\item for every $x\in \mac{X}$ we have $x\le x$;
\item for any $x,y,z\in \mac{X}$, if $x\le y$ and $y\le z$ then $x\le y\le z$.
\end{enumerate}
We say two elements $x,y$ of $\mac{X}$ are \emph{comparable} if $x\le y$ or $y\le x$.

We say 
$\le$ is a \emph{well-quasi-ordering} on $\mac{X}$, or $\mac{X}$ is \emph{well-quasi-ordered} under $\le$, or $(\mac{X},\le)$ is a \emph{well-quasi-order}, if for every infinite sequence $\{x_n\}_{n\ge0}$ of elements of $\mac{X}$, there are indices $i<j$ satisfying $x_i\le x_j$.

An \emph{antichain} is a subset of $\mac{X}$ having no two distinct comparable elements.
A subclass $S$ of $\mac{X}$ is \emph{closed} by $\le$ if $y\in S$ and $x\le y$ imply $x\in S$. An antichain $\mac{C}$ is called a \emph{forbidden list for $S$ by $\le$} if for all $x\in\mac{X}$, $x$ belongs to $S$ if and only if there is no $y\in\mac{C}$ satisfying $y\le x$. When $\mac{X}$ is a class of graphs, $\mac{X}$ is \emph{hereditary} if $\mac{X}$ is closed under induced subgraphs; that is, if $G\in\mac{X}$ and $H$ is isomorphic to an induced subgraph of $G$ then $H\in\mac{X}$.

\section{An equivalence relation involving cut-rank functions}\label{sec:equiv}
An \emph{attached star} in a graph $G$ is an induced subgraph isomorphic to a star whose noncentral vertices are leaves in $G$. In other words, an attached star in $G$ is an induced subgraph, say $G[S]$, isomorphic to a star such that the set of noncentral vertices is anticomplete to $V(G)\setminus S$. The \emph{size} of an attached star is the number of its vertices.

In $G$, let $\equiv_G$ be a binary relation on $V(G)$ such that for $x,y\in V(G)$, $x\equiv_G y$ if $\rho_G(\{x\})=\rho_G(\{y\})\ge\rho_G(\{x,y\})$. It is easy to see that $x\equiv_G y$ if and only if one of the following~holds:
\begin{enumerate}[(i)]
\item $x$ and $y$ are twins in $G$, or
\item one of them is a leaf in $G$ whose unique neighbor is the other.
\end{enumerate}
Furthermore, $\equiv_G$ is in fact an equivalence relation on $V(G)$, as shown by the following. 
\begin{proposition}\label{prop:equivcut}The relation $\equiv_G$ is an equivalence relation on $V(G)$. Moreover, each equivalence class of $(V(G),\equiv_G)$ is one of the following types in $G$: the vertex set of an attached star, a clique of true twins, and an independent set of false twins.
\end{proposition}
\begin{proof}By definition, it is obvious that for $x,y\in V(G)$, $x\equiv_G x$, and if $x\equiv_G y$, then $y\equiv_G x$. Thus $\equiv_G$ is reflexive and symmetric.
To prove that $\equiv_G$ is an equivalence relation on $V(G)$, it remains to show that $x\equiv_G y$ and $y\equiv_G z$ imply $x\equiv_G z$. We may assume that $x$, $y$, $z$ are distinct. 
We have three cases to consider.
\begin{enumerate}
\item $d_G(y)=0$. We have $x,y,z$ are isolated in $G$ and $x\equiv_Gz$.
\item $d_G(y)=1$. If $N_G(y)\not\subseteq\{x,z\}$ then trivially $x,y,z$ are leaves in $G$ with a unique common neighbor, so $x\equiv_Gz$. If $N_G(y)\subset\{x,z\}$, we may assume that $xy\in E(G)$ and $yz\nin E(G)$, then $y,z$ are twins in $G$ which implies that $z$ is a leaf in $G$ whose unique neighbor is $x$, and thus $x\equiv_Gz$.
\item $d_G(y)\ge2$. Now, 
if $x$ is a leaf in $G$, then $y$ is the unique neighbor of $x$ in $G$ and so $x$ is non-adjacent to $z$, which implies $z$ is a leaf whose unique neighbor is $y$ because $y\equiv_Gz$ and therefore $x\equiv_G z$.
By symmetry, if $z$ is a leaf in $G$, then $x\eq_G z$.
If neither $x$ nor $z$ is a leaf in $G$, then $\{x,y\}$ and $\{y,z\}$ are two pairs of twins in $G$, thus $x$ and $z$ are also twins in $G$, and so $x\equiv_Gz$.
\end{enumerate}

Now let $C$ be an equivalence class in $(V(G),\equiv_G)$.
If there is a vertex $x$ in $C$ which is a leaf in $G$ then its unique neighbor, say $y$, must also be in $C$;
so every vertex in $C\setminus\{x,y\}$, being a twin of~$x$,
is also a leaf in $G$ whose unique neighbor is $y$, which implies that $G[C]$ is an attached star in~$G$.
On the other hand, if $C$ contains no leaves in $G$, then necessarily they are pairwise twins.
It is well known that a set of pairwise twins is either a clique of true twins or an independent set of false twins in $G$.
So, we conclude that $C$ is the vertex set of an attached star, a clique of true twins, or an independent set of false twins.
\end{proof}

In addition, an immediate consequence of Proposition~\ref{prop:cutrank} is that local complementations preserve every equivalence class in $(V(G),\equiv_G)$.
\begin{corollary}\label{cor:equivlocal}For every $x,y,v\in V(G)$, $x\equiv_G y$ if and only if $x\equiv_{G*v}y$. In other words, the equivalence classes of $(V(G),\equiv_G)$ remain unchanged in $(V(G),\equiv_{G*v})$.
\end{corollary}

\section{Average cut-rank}\label{sec:basic}
The \emph{average cut-rank} of $G$ is defined as
\[
\ero(G):=\frac{1}{2^{|G|}}\sum_{S\subseteq V(G)}\rho_G(S).
\]
In other words, $\ero(G)$ is the expected value of $\rho_G(S)$ where $S$ is chosen uniformly at random among all subsets of $V(G)$. Note that due to the symmetry of $\rho_G$, $\ero(G)$ is a rational number whose denominator in closed form is a positive integer dividing $2^{|G|-1}$.

One of the reasons to study the average cut-rank is that it does not increase 
when taking vertex-minors. The following theorem not only shows this but also shows that the average cut-rank strictly decreases whenever we take a vertex-minor except for some trivial cases.
\begin{theorem}\label{thm:avgvertexminor}
  If $H$ is a vertex-minor of a graph $G$, then 
  \[\ero(H)\le \ero(G).\]
  In addition, if 
  $V(G)\setminus V(H)$ has at least one non-isolated vertex, 
  then 
  \[\ero(H)\le \ero(G)-2^{-\abs{H}}.\]
\end{theorem}

\begin{proof}Since $H$ is a vertex-minor of $G$, $H$ is an induced subgraph of some graph $G'$ which is locally equivalent to~$G$.
    By Proposition~\ref{prop:cutrank},
    $\ero(G)=\ero(G')$.
Because an isolated vertex remains isolated after each local complementation, we may assume that $H$ is an induced subgraph of $G$.
Let $S$ be a subset of $V(G)$ chosen uniformly at random.
  So $S\cap V(H)$ is a random subset of $V(H)$.
  We have
\[ \rho_G(S)=\rank(A_{G}[S,V(G)\setminus S])\ge \rank(A_{H}[S\cap V(H),V(H)\setminus S)=\rho_H(S\cap V(H)), \]
so $\ero(G)=\mab{E}_S\rho_{G}(S)\ge \mab{E}_{T}\rho_H(T)=\ero(H)$. 
  
Suppose that $V(G)\setminus V(H)$ has at least one non-isolated vertex. 
If there is some vertex in $V(G)\setminus V(H)$, say $v$, having at least one neighbor in $V(H)$, then any subset $S$ of $V(G)\setminus V(H)$ containing~$v$ or any subset $S$ of $V(G)$ containing $V(H)$ but not~$v$ satisfies $\rho_{G}(S)\ge 1$ while $\rho_H(S\cap V(H))=0$.

If no vertex in $V(G)\setminus V(H)$ has a neighbor in $V(H)$, then $G\setminus V(H)$ has at least one edge, say $uv$ for $u,v\in V(G)\setminus V(H)$. Then any subset $S$ of $V(G)$ such that $S$ contains only one of $u,v$ and $S\cap V(H)$ is $\emptyset$ or $V(H)$ satisfies $\rho_G(S)\ge1$ and $\rho_G(S\setminus U)=0$.

In both cases, we have
\[ \ero(G)\ge \ero(H)+\frac{2\cdot 2^{|G|-|H|-1}}{2^{|G|}}=\ero(H)+2^{-|H|}.\qedhere\]
\end{proof}

As an example, we compute the average cut-rank of complete graphs and complete bipartite graphs. We omit its easy proof.
\begin{lemma}\label{lem:avgcliques}For integers $m,k\ge 1$,
\[\ero(K_k)=1-2^{1-k},\quad \ero(K_{m,k})=\frac{(2^m-1)(2^k-1)}{2^{m+k-1}}.\]
In particular $\ero(K_{1,k})=1-2^{-k}=\ero(K_{k+1})$.
\end{lemma}

The following result shows that $1-2^{1-|G|}$ is in fact the smallest possible average cut-rank of any graph $G$ with no isolated vertices. The equality holds if $G$ is a complete graph or a star, by Lemma~\ref{lem:avgcliques}.
\begin{proposition}\label{prop:avgsmallest}A graph $G$ without isolated vertices has average cut-rank at least $1-2^{1-|G|}$. The equality holds if and only if $G$ is a star or a complete graph.
\end{proposition}
\begin{proof}If $G$ is connected, then for every nonempty proper subset $S$ of $V(G)$, $G[S,V(G)\setminus S]$ has at least one edge, hence $\rho_G(S)\ge1$. Because there are $2^{|G|}-2$ subsets $S$ of this type, we obtain
\[ \ero(G)=\frac{1}{2^{|G|}}\sum_{S\subseteq V(G)}\rho_G(S)\ge \frac{2^{|G|}-2}{2^{|G|}}=1-2^{1-|G|}. \]
If $G$ is disconnected, then since $G$ has no isolated vertices, $G$ contains an induced subgraph isomorphic to $2K_2$. It follows that, by Theorem~\ref{thm:avgvertexminor},
\[ \ero(G)\ge \ero(2K_2)=1>1-2^{1-|G|}. \]

Now we consider the equality case. The preceding argument shows that if $\ero(G)=1-2^{1-|G|}$, then $G$ is necessarily connected and $\rho_G(S)=1$ for all nonempty proper subsets $S$ of $V(G)$. In particular, it follows that for all $x,y\in V(G)$, we have $\rho_G(\{x\})=\rho_G(\{y\})=\rho_G(\{x,y\})=1$, or equivalently $x\equiv_G y$. Therefore, $(V(G),\equiv_G)$ has only one equivalence class, so Proposition~\ref{prop:equivcut} implies that $G$ is a star, a complete graph, or an edgeless graph. Because $G$ is connected, $G$ is thus a star or a complete graph. Lemma~\ref{lem:avgcliques} then completes the proof.
\end{proof}
Theorem~\ref{thm:avgvertexminor} provides a lower bound on $\ero(G)-\ero(H)$
when $H$ is a vertex-minor of $G$. 
The next proposition gives an upper bound on this difference.

\begin{proposition}\label{prop:avgdecompose}Let $G_1$ and $G_2$ be graphs and $G=G_1\triangle G_2$. Then
\[ \ero(G)\le \ero(G_1)+\ero(G_2), \]
and the equality holds if (but not necessarily only if) $V(G_1)\cap V(G_2)=\emptyset$, i.e. $G$ is the disjoint union of $G_1$ and $G_2$. In particular, for every vertex $v\in V(G)$,
\[ \ero(G-v)\ge \ero(G)-1+2^{-d_G(v)}\ge \ero(G)-1+2^{1-|G|}. \]
\end{proposition}
\begin{proof}For $i=1,2$, let $H_i$ be the graph with vertex set $V(G)$ and edge set $E(G_i)$. Then $G=H_1\triangle H_2$ and by Theorem~\ref{thm:avgvertexminor}, $\ero(H_i)=\ero(G_i)$ for $i=1,2$.
Choose a subset $S$ of $V(G)$ uniformly at random and set $T:=V(G)\setminus S$. Then since $G=H_1\triangle H_2$, we have
\begin{align*}
\rho_G(S)=\rank(A_G[S,T])&=\rank(A_{H_1}[S,T]+A_{H_2}[S,T])\\
&\le \rank(A_{H_1}[S,T])+\rank(A_{H_2}[S,T])=\rho_{H_1}(S)+\rho_{H_2}(S). 
\end{align*}
This implies immediately that
\begin{equation*}
\ero(G)\le \ero(H_1)+\ero(H_2)=\ero(G_1)+\ero(G_2).
\end{equation*}
If $V(G_1)\cap V(G_2)=\emptyset$, for $i=1,2$ let $S_i$ be a random subset of $V(G_i)$ and let $T_i:=V(G_i)\setminus S_i$. Then $S_1\cup S_2$ is a random subset of $V(G)$ and $T_1\cup T_2=V(G)\setminus(S_1\cup S_2)$. It is easy to see that
\begin{equation*}
\rho_{G_1}(S_1)+\rho_{G_2}(S_2)=\rank(A_{G_1}[S_1,T_1])+\rank(A_{G_2}[S_2,T_2])=\rank(A_G[S_1\cup S_2,T_1\cup T_2]). 
\end{equation*}
As a result, we deduce $\ero(G)=\ero(G_1)+\ero(G_2)$.

Now for any vertex $v\in V(G)$, let $G_1:=G-v$ and $G_2:=G[\{v\},N_G(v)]$, which is isomorphic to $K_{1,d_G(v)}$. By Lemma~\ref{lem:avgcliques} we obtain 
\begin{equation*}
\ero(G)\le \ero(G-v)+\ero(K_{1,d_G(v)})=\ero(G-v)+1-2^{-d_G(v)}\le \ero(G-v)+1-2^{1-|G|}.\qedhere 
\end{equation*}
\end{proof}

It can be seen that the lower bound from Theorem~\ref{thm:avgvertexminor} and the upper bound from Proposition~\ref{prop:avgdecompose} are pretty far apart, because they apply for all graphs in general. In many cases, we need the upper and lower bounds on $\ero(G)-\ero(H)$ to be close enough, when $H$ is a vertex-minor of $G$. The next propositions provide a tighter upper bound compared to Proposition~\ref{prop:avgdecompose} and a tighter lower bound compared to Theorem~\ref{thm:avgvertexminor}, when we have distinctive structures involving false twins, in particular the attached stars.
\begin{proposition}\label{prop:falsetwins}Let $G$ be a graph in which $u_1,\ldots,u_k$ are pairwise false twins, where $k\ge1$. Let $d:=d_G(u_1)$. Then
\[
\ero(G-u_1)\ge\ero(G)-\frac{2^{d}-1}{2^{k+d-1}}.\]
In particular if $d_G(u_1)=1$ then
\[ \ero(G-u_1)\ge\ero(G)-2^{-k}. \]
\end{proposition}
\begin{proof}Let $n:=|G|$, $H:=G-u_1$, $V:=V(G)$, and
\[ \mac{F}:=\{S\subseteq V:\{u_1\}\subsetneq S\cap\{u_1,\ldots,u_k\}\text{ or }\{u_1\}\cup N_G(u_1)\subseteq S\}. \]
Observe that for $S\in\mac{F}$, we have $\rho_G(S)=\rho_H(S\setminus\{u_1\})$, because if $\{u_1\}\subsetneq S\cap\{u_1,\ldots,u_k\}$ then there is some $j\in\{2,\ldots,k\}$ such that $u_j\in S$ and so the row vectors corresponding to $u_1$ and $u_j$ in $A_G[S,V\setminus S]$ are the same, and if $u_1\cup N_G(u_1)\subseteq S$ then the row vector corresponding to $u_1$ in $A_G[S,V\setminus S]$ is zero. Obviously, $\rho_G(S)\le \rho_H(S\setminus\{u_1\})+1$ for all $S\subseteq V$. Therefore, because $\rho_G$ is symmetric and there are exactly $ 2^{n-d-k}(2^{d}-1) $ subsets $S$ of $V$ such that $u_1\in S\not\in\mac{F}$, we have
\[
\begin{aligned}
	\ero(G)&=\mab{E}(\rho_G(S)\mid u_1\in S\subset V)\\
	&=\mab{E}(\rho_G(S)\mid u_1\in S\in\mac{F})\cdot\mab{P}(S\in\mac{F}\mid u_1\in S)+\mab{E}(\rho_G(S)\mid u_1\in S\not\in \mac{F})\cdot\mab{P}(S\not\in\mac{F}\mid u_1\in S)\\
	&
	\le\mab{E}(\rho_H(S\setminus\{u_1\})\mid u_1\in S\in\mac{F})\cdot\mab{P}(S\in\mac{F}\mid u_1\in S)\\ 
	&\quad+\mab{E}(\rho_H(S\setminus\{u_1\})+1\mid u_1\in S\not\in \mac{F})\cdot\mab{P}(S\not\in\mac{F}\mid u_1\in S) 
	\\
	&
	\le\mab{E}(\rho_H(S\setminus\{u_1\})\mid u_1\in S\in\mac{F})\cdot\mab{P}(S\in\mac{F}\mid u_1\in S)\\ 
	&\quad+\mab{E}(\rho_H(S\setminus\{u_1\})\mid u_1\in S\not\in \mac{F})\cdot\mab{P}(S\not\in\mac{F}\mid u_1\in S)+\mab{P}(S\not\in\mac{F}\mid u_1\in S)  
	\\
	&=\mab{E}(\rho_H(S\setminus\{u_1\})\mid u_1\in S\subset V)+\mab{P}(S\not\in\mac{F}\mid u_1\in S)\\
	&=\mab{E}(\rho_H(T)\mid T\subset V(H))+\frac{2^{n-d-k}(2^d-1)}{2^{n-1}}\\
	&=\ero(H)+\frac{2^d-1}{2^{k+d-1}},
\end{aligned}
\]
which completes the proof of the proposition.
\end{proof}
\begin{proposition}\label{prop:attachedstar}Let $G$ be a graph on $n\ge1$ vertices, $T$ be the vertex set of an attached star in~$G$, and $H:=G-T$. Then
\[
\ero(G)-1<\ero(H)\le \ero(G)-1+2^{1-|T|}.
\]
\end{proposition}
\begin{proof}Let $V:=V(G)$ and let $T=\{u_1,\ldots,u_k,v\}$ where $v$ is the central vertex and $u_1,\ldots,u_k$ are leaves of $G[T]$. 
The left hand side inequality is trivial by Proposition~\ref{prop:avgdecompose} and Lemma~\ref{lem:avgcliques}, because
\[ \ero(G)\le \ero(H)+\ero(F)=\ero(H)+1-2^{-d_G(v)}<\ero(H)+1, \]
where $F$ is
the connected subgraph of $G$ consisting of all edges incident with $v$.

We move on to the right hand side inequality. Observe that for every $S\subset V$, we have $\rho_G(S)\ge \rho_H(S\setminus T)$. If furthermore $v\in S$ and $T\not\subset S$, then in $A_G[S,V\setminus S]$, the row vectors corresponding to $T\cap S\setminus \{v\}$ are all zero vectors, and for every $u_j\in T\setminus S$, the column vector corresponding to $u_j$ has only one $1$ as its common entry with the row vector corresponding to $v$. It follows that in $A_G[S,V\setminus S]$, this row vector is linearly independent to the other row vectors, and in $A_{G}[S\setminus\{v\},V\setminus S]$, the column vectors corresponding to $T\setminus S$ are all zero vectors. Hence $\rho_G(S)=\rho_H(S\setminus T)+1$. Therefore, due to the symmetry of $\rho_G$,
\[
\begin{aligned}
	\ero(G)&=\mab{E}(\rho_G(S)\mid v\in S\subset V)\\ 
	&=\mab{E}(\rho_G(S)\mid v\in T\subset S)\cdot\mab{P}(T\subset S\mid v\in S)+\mab{E}(\rho_G(S)\mid v\in S,T\not\subset S)\cdot\mab{P}(T\not\subset S\mid v\in S)\\
	&
	\ge\mab{E}(\rho_H(S\setminus T)\mid v\in T\subset S)\cdot\mab{P}(T\subset S\mid v\in S)\\ 
	&\quad+\mab{E}(\rho_H(S\setminus T)+1\mid v\in S,T\not\subset S)\cdot\mab{P}(T\not\subset S\mid v\in S)
	\\
	&
	\ge\mab{E}(\rho_H(S\setminus T)\mid v\in T\subset S)\cdot\mab{P}(T\subset S\mid v\in S)\\ 
	&\quad+\mab{E}(\rho_H(S\setminus\{v\})\mid v\in S,T\not\subset S)\cdot\mab{P}(T\not\subset S\mid v\in S)+\mab{P}(T\not\subset S\mid v\in S)  
	\\
	&=\mab{E}(\rho_H(S\setminus T)\mid v\in S\subset V)+\mab{P}(T\not\subset S\mid v\in S)\\
	&=\mab{E}(\rho_H(U)\mid U\subset V(H))+\frac{2^{n-k-1}(2^k-1)}{2^{n-1}}\\
	&=\ero(H)+1-2^{1-|T|}.
\end{aligned}
\]
This completes the proof.
\end{proof}

\section{Characterization of classes of graphs of bounded average cut-rank}\label{sec:wqo}
In this section, we will prove Theorem~\ref{thm:nd}, which characterizes classes of graphs of bounded average cut-rank and relates them to existing concepts. We will also discuss some corollaries on well-quasi-ordering.

We start with some definitions solely used in this section. In a graph $G$, two vertices $x,y$ are called \emph{twin-equivalent} if either $x=y$ or they are twins. 
It is easy to verify
that the relation \dd{}twin-equivalent\ee{} is an equivalence relation on $V(G)$.
Thus the vertex set of $G$ can be partitioned into \emph{twin classes}.
The \emph{neighborhood diversity} of $G$, first defined by Lampis~\cite{lampis2012}, is the number of twin classes in $G$.

Here is a fundamental property on the rank and the number of distinct rows of a $0$-$1$ matrix.
\begin{lemma}\label{lem:rows}Any $0$-$1$ matrix $M$ has at most $2^{\rank(M)}$ distinct rows.
\end{lemma}
\begin{proof}
  Let $r=\rank(M)$.
  Then $M$ has a non-singular $r\times r$ submatrix, whose columns are indexed by $I$. Note that $\abs{I}=r$ and 
  each row vector
  is completely determined by the $0$-$1$ values on the entries in $I$
  and therefore $M$ has at most $2^{\abs{I}}$ distinct rows.
\end{proof}

The authors would like to thank Alex Scott (personal communication) for suggesting the proof of the following lemma.
\begin{lemma}\label{lem:avgmax}
  For every graph $G$, 
  $\mco(G)\le 4\ero(G)$.
\end{lemma}
\begin{proof}
  Let $k$ be the maximum cut-rank of $G$. Then there are two disjoint subsets $A$, $B$ of $V(G)$ satisfying $|A|=|B|=k$ and $\rho_G^*(A,B)=k$. Let $H:=G[A\cup B]$.
  Since $\ero(G)\ge \ero(H)$, it suffices to show that $\ero(H)\ge k/4$.

Let $S$ be a subset of $V(H)=A\cup B$ chosen uniformly at random. Then $S\cap A$ is a random subset of $A$, and  $B\setminus S$ is a random subset of $B$. Since
\[ \rho_H(S)=\rho_H^*(S,(A\cup B)\setminus S)\ge \rho_H^*(S\cap A,B\setminus S), \]
we have
\[ \ero(H)=\mab{E}_S\rho_G(S)\ge \mab{E}_S\rho_H^*(S\cap A,B\setminus S)=\mab{E}_{X,Y}\rho_H^*(X,Y), \]
where the last expression indicates the expected value of $\rho_H^*(X,Y)$ with $X,Y$ selected from $2^A$, $2^B$, respectively.

Fix $X\subseteq A$ and let $Y\subseteq B$ be random. Because $\rho_H^*(A,B)=\rho_G^*(A,B)=k$ and $|A|=|B|=k$, $\rho_H^*(X,B)=|X|$, so there is a subset $Z$ of $B$ satisfying $|Z|=|X|$ and $A_H[X,Z]$ has full rank. Because $Y\subseteq B$ is random, $Y\cap Z$ is a random subset of $Z$, which implies that
\[ \mab{E}_Y\rho_H^*(X,Y)\ge \mab{E}_Y\rho_H^*(X,Y\cap Z)=\mab{E}_Y|Y\cap Z|=\frac{|X|}{2}. \]
Therefore
\[ \mab{E}_{X,Y}\rho_G^*(X,Y)=\mab{E}_{X}\mab{E}_Y\rho_G^*(X,Y)\ge\mab{E}_X\frac{|X|}{2}=\frac{k}{4}. \]
Thus $\ero(H)\ge k/4$ and the conclusion follows.
\end{proof}

The following lemma shows that a hereditary class of graphs is of bounded maximum cut-rank (or average cut-rank) if and only if it is of bounded neighborhood diversity.
This result is also essential to the proof of Theorem~\ref{thm:main2}.
\begin{lemma}\label{lemma:key}%
    For every graph $G$, 
    $\nd(G)< 2^{2\mco(G)+2}$.
\end{lemma}
\begin{proof}[Proof (Adapted from the proof of Lemma 4.5 in~\cite{dko2019})]
	Let $A$ be a maximal subset of $V(G)$ without any pair of twins in $G$.
	We construct a complete graph $H$ on the vertex set $A$ and label every edge $uv$ of $H$ as follows:
	$uv$ is labeled by $w$ for some $w\in V(G)\setminus\{u,v\}$ adjacent to only one among $u$ and $v$ in $G$.
	This labeling exists because of the definition of $A$.
	Let $m:=|A|$, and let $S$ be a random subset of $A=V(H)$ where each vertex is included independently at random with probability $p:=1/\sqrt{m}$.
	For every edge $e\in E(H)$, let $X_e$ be the indicator random variable for the event that the ends and the label of $e$ in $H$ are in $S$, and put $X:=\sum_{e\in E(H)}X_e$.
	Then for all $e\in E(H)$, $\mathbb E[X_e]=p^3$ if the label of $e$ is in $A$ and $\mathbb E[X_e]=0$ otherwise.
  	By linearity of expectation
	\[ \mathbb E[|S|-X]=\mathbb E[|S|]-\mathbb E[X]\ge pm-p^3{m\choose 2}> pm-\frac{p^3m^2}{2}=\frac{\sqrt{m}}{2}. \]
	Thus, there is a subset $S$ of $A$ such that $|S|-X>\sqrt{m}/2$;
	that is, there are fewer than $|S|-\sqrt{m}/2$ edges in $H$ having ends and labels in $S$.
	Then, by deleting one end for each such edge, we get a subset $T$ of $S$ satisfying $|T|> \sqrt{m}/2$ and for every distinct $u,v\in T$,
	in $H$ the label of $uv$ does not belong to $T$.
	This means that for every distinct $u,v\in T$, in $G$ there is a vertex $w$ outside $T$ which is adjacent to only one of $u$ and $v$,
	which implies that $A_G[T,V(G)\setminus T]$ has more than $\sqrt{m}/2$ distinct rows.
	Hence, by Lemma~\ref{lem:rows},
	\[ 2^{\mco(G)}\ge 2^{\rho_G(T)}> \frac{\sqrt{m}}{2}, \]
	which implies $|A|=m< 2^{2\mco(G)+2}$.
	As every vertex in $V(G)\setminus A$ is a twin of some vertex in $A$ ($A$ is maximal),
	we conclude that $V(G)$ can be partitioned into less than $2^{2\mco(G)+2}$ twin classes.
\end{proof}

Now we are ready to prove Theorem~\ref{thm:nd}.
\thmnd*
\begin{proof}
  As $G$ has at least one edge, $\ero(G)<\mco(G)$ trivially.
  Since $\rho_G(X)\le \mr(\mathbb F_2,G)$ for all $X\subset V(G)$
  trivially,
  $\mco(G)\le \mr(\mathbb F_2,G)$.

  To prove $\mr(\mathbb F,G)\le \nd(G)$ for any field $\mathbb F$, 
  let us assume that $k=\nd(G)$ and so $G$ has exactly $k$ twin classes. 
  Starting from the adjacency matrix of $G$, 
  we change the diagonal entry of a vertex $v$ 
  to $1$ if $v$ belongs to a twin class that is a clique of $G$.
  The resulting matrix has $k$ distinct rows
  and so its rank is at most $k$. 
  This proves that $\mr(\mathbb F,G)\le k$.
  
  Since every matrix of rank $k$ over $\mathbb{F}$ has at most $\abs{\mathbb F}^k$ distinct rows, we have $\nd(G)\le \abs{\mathbb F}^{\mr(\mathbb F,G)}$. This was shown by Ding and Kotlov~\cite[Corollary 2.2]{DK2006}.

  Lemmas~\ref{lem:avgmax} and \ref{lemma:key} show that $\nd(G)<2^{2\mco(G)+2}\le2^{8\ero(G)+2}$.

  Let $t=\cd(G)$. Then  there are complete graphs $G_1,\ldots,G_t$ such that $G=G_1\triangle\cdots\triangle G_t$. As $E(G)\neq\emptyset$, $t\ge 1$.
  By Proposition~\ref{prop:avgdecompose} and Lemma~\ref{lem:avgcliques}, we see that
  $\ero(G)\le\sum_{i=1}^t\ero(G_i)<\cd(G)$.
  Also, $V(G)$ can be partitioned into $2^{t}$ subsets,
  each of them is a set of pairwise twins, based on the inclusion of $V(G_1)$, $V(G_2)$, $\ldots$, $V(G_t)$. This leads to an inequality that $\nd(G)\le 2^{\cd(G)}$.
  
  Now, it remains to prove that $\cd(G)\le \frac{3}{2}\mr(\mathbb F_2,G)$.
  Let $A$ be a symmetric matrix over $\mathbb F_2$ of rank $m$ realizing $\mr(\mathbb F_2,G)$.
  It is known that every symmetric matrix of rank $m$ can be written as a sum of $m-2s$ rank-$1$ symmetric matrices and $s$ rank-$2$ symmetric matrices, see Godsil and Royle~\cite[Lemma 8.9.3]{GR2001}.
  As the field is binary, we can also deduce easily that in the outcome, the rank-$2$ symmetric matrices have zero diagonals, by using the proof of \cite[Lemma 8.10.1]{GR2001}.
  Rank-$1$ symmetric matrices over $\mathbb F_2$ are of the form 
  \[
  \begin{pmatrix}
    1&0\\
    0&0
  \end{pmatrix}
  \]
  where $1$ represents an all-$1$ matrix,
  $0$ represents an all-$0$ matrix,
  and the diagonal entries represent square matrices.
  Thus, every rank-$1$ symmetric matrix over $\mathbb F_2$ is the adjacency matrix of one complete graph with some isolated vertices, while changing a few diagonal entries to $1$.
  Rank-$2$ symmetric matrices over $\mathbb F_2$ with zero diagonals are of the form 
  \[
  \begin{pmatrix}
    0&1&1\\
    1&0&1\\
    1&1&0
  \end{pmatrix}
  \text{ or }
  \begin{pmatrix}
    0&1\\
    1&0
  \end{pmatrix}.
  \]
  and so every rank-$2$ symmetric matrix 
  over $\mathbb F_2$ with zero diagonals can be written as the sum of three rank-$1$ symmetric matrices as follows.
  \[
    \begin{pmatrix}
      0&1&1\\
      1&0&1\\
      1&1&0
    \end{pmatrix}
    =
    \begin{pmatrix}
      1&1&0\\
      1&1&0\\
      0&0&0
    \end{pmatrix}+    \begin{pmatrix}
      0&0&0\\
      0&1&1\\
      0&1&1
    \end{pmatrix}+   \begin{pmatrix}
      1&0&1\\
      0&0&0\\
      1&0&1
    \end{pmatrix}.
    \]
    Thus, $A$ can be written as a sum of at most $3m/2$ rank-$1$ symmetric matrices over $\mathbb F_2$.
    This proves that $\cd(G)\le \frac{3}{2}\mr(\mathbb F_2,G)$.
\end{proof}
Corollary~\ref{cor:wqo} yields the following corollary.
\begin{corollary}\label{cor:finitelist}
  Let $\mathcal C$ be a hereditary class of graphs.
  If graphs in $\mathcal C$ have bounded average cut-rank,
  then there exists a finite list of graphs $H_1$, $H_2$, $\ldots$, $H_k$
  such that
  a graph $G$ is in $\mathcal C$
  if and only if
  $G$ has no induced subgraph isomorphic to $H_i$ for every $i=1,\ldots,k$. 
\end{corollary}
\begin{proof}
  Let $\alpha$ be a real such that every graph in $\mathcal C$ has average cut-rank at most~$\alpha$.
  If $H$ is an induced-subgraph-minimal graph not in $\mathcal C$,
  then $H$ has average cut-rank at most $\alpha+1$
  by Proposition~\ref{prop:avgdecompose}.
  By Corollary~\ref{cor:wqo}, there are only finitely many 
  induced-subgraph-minimal graphs not in $\mathcal C$,
  because they form an antichain.
\end{proof}

Let $S_{\ero}$ be the set of all reals $\alpha$ such that there exists a graph with average cut-rank $\alpha$. By the definition of average cut-rank, this set is trivially a subset of $\{p/2^q:p\in\mab{N}\cup\{0\},q\in\mab{N}\}$. By the previous corollary, we deduce the following topological property of $S_{\ero}$.
\begin{proposition}For any $\alpha\ge0$ there is some $\delta_{\alpha}>0$ such that every graph has average cut-rank outside $(\alpha,\alpha+\delta_{\alpha})$. This implies that $S_{\ero}$ is not dense in any interval, hence is nowhere dense in $[0,\infty)$.
\end{proposition}
\begin{proof}
  By Corollary~\ref{cor:finitelist}, there exists a finite list $\{G_1,\ldots,G_m\}$ of forbidden induced subgraphs for the class of graphs of average cut-rank at most $\alpha$. Because $\alpha<\ero(G_j)=:r_j$ for all $j=1,\ldots,m$, we have $\alpha<\min\{r_1,\ldots,r_m\}=:q_{\alpha}$. Hence there is no graph having average cut-rank lying inside $(\alpha,q_{\alpha})$. The conclusion thus follows for $\delta_{\alpha}=q_{\alpha}-\alpha$.
\end{proof}

\section{Upper bound on the size of induced subgraph obstructions}\label{sec:upperbound}

Ding and Kotlov~\cite{DK2006} proved that 
each forbidden induced subgraph for the class of graphs of 
minimum rank over a finite field $\mathbb F$ at most $k$
has at most $(\abs{\mathbb F}^k/2+1)^2$ vertices.

We can find an upper bound on 
the size of  
each forbidden induced subgraph for the class of graphs of 
maximum cut-rank at most $k$ as follows.

\begin{theorem}
  \label{thm:maxcutrankobs}
  If $\mco(G)>k$ and
  $\mco(G-v)\le k$ for all vertices $v$ of $G$, then 
  $\abs{G}=2k+2$.
\end{theorem}
\begin{proof}
  If $\mco(G)>k$, then there exists a pair $(X,Y)$ of disjoint sets of vertices
  such that $\abs{X}=\abs{Y}=k+1$ and 
  the rank of $A(G)[X,Y]=k+1$.
  If $\abs{G}>2k+2$, then there is a vertex $v\notin X\cup Y$ and therefore $\mco(G-v)\ge \rank A(G)[X,Y]=k+1$, contradicting the assumption. 
  Trivially, if $\abs{G}<2k+2$, then $\mco(G)\le k$.
\end{proof}

Now we will find such an upper bound for the class of
graphs of average cut-rank at most $\alpha$, thus
proving Theorem~\ref{thm:main2}. 
For convenience, we recall the sequence $\{x_n(\vep)\}_{n\ge0}$ defined in Section~\ref{sec:introduction}, as follows.
\begin{align*} 
x_0(\vep)&=\max(\lfloor2-\log(1-\vep)\rfloor,5),\\
x_{n}(\vep)&=2^{8n+10}\lfloor x_{n-1}(\vep)-\log (1-\{2^{x_{n-1}(\vep)}\vep/2\})+1\rfloor
\quad\text{ for all integers }n\ge1.
\end{align*}

\sizeobstruction*
\begin{proof}%
  Let $\alpha=\vep+n$ where $\vep=\{\alpha\}\in[0,1)$ and $n=\lfloor\alpha\rfloor\in\mab{N}\cup\{0\}$. We fix $\vep$ and proceed by induction on $n\ge0$. For convenience, set $x_n:=x_n(\vep)$ for all $n\ge0$.

  First let us assume that $n=0$.
  If there is a vertex $v\in V(G)$ such that $G-v$ has no isolated vertices,
	then by Proposition~\ref{prop:avgsmallest},
	$\vep\ge \ero(G-v)\ge 1-2^{1-(|G|-1)}$,
  which implies that $\abs{G}\le \lfloor2-\log(1-\vep)\rfloor\le x_0$.
  Thus we may assume that the deletion of every vertex of $G$ yields a graph with some isolated vertex.
  It follows that $E(G)$ is a perfect matching
  and therefore $\vep\ge\ero(G-v)=\ero(\frac{\abs{G}-1}{2}K_2)=\frac{\abs{G}-1}{4}$.
  This implies that $\abs{G}\le 4\vep+1<5\le x_0$.

  Now we may assume that $n>0$. 
	Suppose for the sake of contradiction that $\abs{G}\ge x_{n}$.
  Observe that by Theorem~\ref{thm:nd}, for any vertex $v$, 
  $\nd(G)\le 2\nd(G-v)+1<2\cdot 2^{8(n+\vep)+2}+1\le x_n\le \abs{G}$ and therefore
  there is a vertex $v$ having a twin.
  Then $\nd(G)=\nd(G-v)< 2^{8(n+\vep)+2}<2^{8n+10}$
  and therefore $G$ has a twin class $C$ with $\abs{C}>\abs{G}/2^{8n+10}$.

  Note that $\abs{C}>x_{n-1}\ge x_0\ge 5$.   Let $x$, $z$ be distinct vertices in $C$.
	\begin{itemize}
		\item If $C$ is a clique of true twins in $G$,
		then 
    $(G*x)[C]$ is an attached star in $G*x$.
    Let $G':=G*x$.
		\item If $C$ is an independent set of false twins in $G$,
		then since the vertices in $C$ are nonisolated in~$G$,
		there is some $y\in N_G(C,V(G)\setminus C)$.
    Then $C$ is a clique of true twins in $G*y$
    and $(G*y*x)[C]$ is an attached star in $G*y*x$.
    Let $G':=G*y*x$.
	\end{itemize}
	Let $S:=C\setminus\{z\}$, $H:=G-z$, and $H':=G'-z$. 
  Then in both cases, $H'$ is locally equivalent to~$H$,
  $H'[S]$ is an attached star in $H'$, and $G'-C=H'-S$.
	By Proposition~\ref{prop:attachedstar} and Theorem~\ref{thm:avgvertexminor}
	we deduce that
	\[ \ero(H)=\ero(H')\ge \ero(H'-S)=\ero(G'-C)>\ero(G')-1=\ero(G)-1\ge \vep+n-1, \]
	thus $H'-S$ contains some induced-subgraph-minimal graph of average cut-rank larger than $\vep+n-1$,
  say $F$, as an induced subgraph.
  Note that $F$ has no isolated vertices because deleting isolated vertices does not change the average cut-rank.
	By the induction hypothesis, $F$ has less than $x_{n-1}$ vertices.
	Then, $\ero(F)$ is a rational number larger than $\vep+n-1$ whose denominator divides $2^{x_{n-1}-1}$,
	so by Theorem~\ref{thm:avgvertexminor} we see that
	\[ \ero(H'-S)\ge\ero(F)\ge\vep+n-1+\frac{1-\{2^{x_{n-1}}\vep/2\}}{2^{x_{n-1}}/2}. \]
	By Theorem~\ref{thm:avgvertexminor} and Proposition~\ref{prop:attachedstar}, we thus obtain
	\begin{align*}
    \vep+n\ge \ero(H)=\ero(H')&\ge \ero(H'-S)+1-2^{1-|S|}\\
    & \ge \vep+n-1 +\frac{1-\{2^{x_{n-1}}\vep/2\}}{2^{x_{n-1}}/2}+ 1-2^{1-|S|}.
	\end{align*}
  Thus, we deduce that 
  \[
  1-\abs{S}\ge -x_{n-1}+1+\log (1-\{2^{x_{n-1}}\vep/2\})
  \]
  and so $\abs{S}\le \lfloor x_{n-1}-\log (1-\{2^{x_{n-1}}\vep/2\})\rfloor$
  and $\abs{C}\le \lfloor x_{n-1}-\log (1-\{2^{x_{n-1}}\vep/2\})+1\rfloor$.
  This is a contradiction because $\abs{C}>\abs{G}/2^{8n+10}\ge x_n/2^{8n+10}$.
\end{proof}

\section{Average cut-rank and forbidden vertex-minors}\label{sec:forbidden}
\subsection{Forbidden vertex-minors}
By Corollary~\ref{cor:finitelist}, 
we can observe the following.
\begin{quote}
  Let $\mac{C}$ be a class of graphs closed under taking vertex-minors. If $\mac{C}$ has bounded average cut-rank, then there exists a finite list of graphs $G_1$, $G_2$, $\ldots$, $G_m$
such that
a graph $G$ is in $\mathcal C$
if and only if
$G$ has no vertex-minor isomorphic to $G_j$ for every $j=1,\ldots,m$.
\end{quote}
A minimal such list is called a \emph{list of forbidden vertex-minors} for $\mathcal C$.
A list of forbidden vertex-minors is not unique, as one can replace a graph in the list with any locally equivalent graph.

But essentially the list is determined up to some equivalence relation.
For two classes $\mac{S}_1$ and $\mac{S}_2$ of graphs, we say that $\mac{S}_1$ is 
\emph{locally equivalent} to $\mac{S}_2$, 
denoted by $\mac{S}_1\simeq\mac{S}_2$, if for every $G\in\mac{S}_1$ there is some $H\in\mac{S}_2$ isomorphic to a graph locally equivalent to $G$ and for every $H\in\mac{S}_2$ there is some $G\in\mac{S}_1$ isomorphic to a graph locally equivalent to $H$. 
Then we can easily verify that the relation $\simeq$ is an equivalence relation and for every class of graphs closed under taking vertex-minors, the list of forbidden vertex-minors for $\mathcal C$ is determined up to local equivalence. As the list is an antichain with respect to the vertex-minor relation, 
every list of forbidden vertex-minors for $\mathcal C$ 
has the same size.

Let $\mac{L}_{\le \alpha}$ be the class of all graphs $H$ satisfying $\ero(H)>\alpha$ and any proper vertex-minor of $H$ has average cut-rank at most $\alpha$, and let $\mac{L}_{<\alpha}$ be the class of all graphs $H$ satisfying $\ero(H)\ge \alpha$ and any proper vertex-minor of $H$ has average cut-rank smaller than $\alpha$. Then by Proposition~\ref{prop:avgdecompose}, 
every graph in $\mac{L}_{\le \alpha}$ or
$\mac L_{<\alpha}$ has average cut-rank smaller than $\alpha+1$.
By Corollary~\ref{cor:wqo}, both $\mac{L}_{\le \alpha}$ and $\mac{L}_{<\alpha}$ are finite.
We can also easily deduce that 
  \begin{quote}
a graph has average cut-rank larger than (or at least) $\alpha$ if and only if it contains a vertex-minor in $\mac{L}_{\le \alpha}$ (or $\mac{L}_{<\alpha}$, respectively).
  \end{quote}
Therefore 
for every $\alpha\ge0$, $\mac{L}_{\le\alpha}$ is locally equivalent to every list of forbidden vertex-minors for the class of graphs of average cut-rank at most $\alpha$.
Similarly, $\mac{L}_{<\alpha}$ is locally equivalent to every list of forbidden vertex-minors for the class of graphs of average cut-rank smaller than $\alpha$.

\subsection{Lower bound on the number of vertex-minor obstructions}
Recall that for every $\alpha\ge0$, every list of forbidden vertex-minors for 
the class of graphs of average cut-rank at most~$\alpha$ is finite and has the same size;
the same happens for the lists of forbidden vertex-minors for 
the class of graphs of average cut-rank smaller than~$\alpha$.
We shall show that, there is some universal constant $c>0$ such that for any $\vep\in[0,1)$ and nonnegative integer $n$, every list of forbidden vertex-minors for the class of graphs of average cut-rank at most (or smaller than) $\vep+n$  contains at least $2^{cn\log (n+1)}$ graphs.
To do so, we construct a set of at least $2^{cn\log (n+1)}$ vertex-minor-minimal graphs of average cut-rank larger than $\vep+n$, such that no two of them are locally equivalent to each other.
Then, we can obtain from this set another set of at least $2^{cn\log (n+1)}$ vertex-minor-minimal graphs of average cut-rank at least $\vep+n$ such that no two of them are locally equivalent to each other.
Let us start with several notions to make our arguments clearer.

For a graph $G$, let $\pi(G)$ denote the quotient graph of $G$ induced by $\equiv_G$. It is not difficult to see that a graph $F$ without isolated vertices is a forest if and only if $\pi(F)$ is a forest and every equivalence class of $(V(F),\equiv_F)$ induces an attached star in $F$. In this case, let $R(F)$ be the set of central vertices in the equivalence classes of $(V(F),\equiv_F)$. Then it is not difficult to check that $F[R(F)]$ is isomorphic to $\pi(F)$. We regard $\pi(F)$ as a weighted graph by assigning each vertex $C$ of $\pi(F)$ the weight $|C|$.

For two forests $F_1$ and $F_2$ without isolated vertices, we shall write $\pi(F_1)\cong \pi(F_2)$ if there is an isomorphism keeping weights from $\pi(F_1)$ to $\pi(F_2)$. From the definitions we can deduce the following easily.
\begin{lemma}\label{lem:idenforests}
  Two forests $F_1$ and $F_2$  without isolated vertices are isomorphic if and only if $\pi(F_1)\cong \pi(F_2)$.
\end{lemma}

The following is another useful characterization of isomorphic forests.
\begin{lemma}[{{Bouchet~\cite[Corollary 5.4]{bouchet1988a}}}]\label{lemma:forests}For two forests $F_1$ and $F_2$, $F_1$ is isomorphic to $F_2$ if and only if $F_1$ is isomorphic to a graph locally equivalent to $F_2$.
\end{lemma}

For two graphs $G$ and $H$, $H$ is called an \emph{elementary} vertex-minor of $G$ if $H$ is a vertex-minor of $G$ and $|H|=|G|-1$. 
The following theorem of Bouchet~\cite{bouchet1988b} characterizes elementary vertex-minors of a graph up to local equivalence. 
Geelen and Oum~\cite{GO2009} provided a direct proof.
\begin{proposition}[{{Bouchet~\cite[Corollary 9.2]{bouchet1988b}}}]\label{prop:elementary}
  Let $v$ be a vertex of a graph $G$.
  If $H$ is a vertex-minor of $G$ with $V(H)=V(G)\setminus\{v\}$, then $H$ is locally equivalent to one of $G-v,(G*v)-v$, and $(G\wedge uv)-v$ for any $u$ adjacent to $v$ in $G$.
\end{proposition}
For a graph $G$, a vertex $v\in V(G)$, and an integer $k\ge0$, we denote by $G+_vK_{1,k}$ the graph obtained from the disjoint union of $G$ and $K_{1,k}$
by adding an edge between $v$ and the central vertex of $K_{1,k}$.
The following lemma is crucial for our construction.
\begin{lemma}\label{lemma:attach}
  Let $G\in\mac{L}_{\le\vep+n}$ and $d\ge1$ be the size of the largest attached star in $G$. Then there exists a unique positive integer $q_1=q_1(G)$ such that $G+K_{1,q_1}\in\mac{L}_{\le \vep+n+1}$ and $q_1\ge d$.
  Furthermore, for each $v\in V(G)$, there exists a unique positive integer $q_2=q_2(G,v)\in\{q_1-1,q_1\}$ such that $G+_vK_{1,q_2}\in\mac{L}_{\le\vep+n+1}$.
\end{lemma}
\begin{proof}First, we prove that
\begin{equation}\label{eq:16}
\ero(G)\le \vep+n+2^{1-d}\le\vep+n+1.
\end{equation}
Indeed, if $d=1$ then for any $u\in V(G)$ we have, by Proposition~\ref{prop:attachedstar},
\[ \ero(G)< \ero(G-u)+1\le \vep+n+1=\vep+n+2^{1-d}. \]
If $d>1$, then let $u$ be a leaf in an attached star of size $d$ in $G$. By Proposition~\ref{prop:falsetwins} and the fact that $G\in\mac{L}_{\le\vep+n}$, we have
\begin{equation*}
 \ero(G)\le \ero(G-u)+2^{1-d}\le \vep+n+2^{1-d}\le \vep+n+1,
\end{equation*}
and \eqref{eq:16} is proved. Hence, because $\ero(G)>\vep+n$, by Lemma~\ref{lem:avgcliques} and Proposition~\ref{prop:avgdecompose},
there is some $q_1\ge1$ such that for all $k\ge q_1$
\begin{equation}\label{eq:17}
\ero(G+K_{1,k})=\ero(G)+1-2^{-k}>\vep+n+1,
\end{equation}
and for all $1\le k<q_1$
\begin{equation}\label{eq:18}
 \ero(G+K_{1,k})=\ero(G)+1-2^{-k}\le \ero(G)+1-2^{1-q_1}\le \vep+n+1.
\end{equation}
Thus, since (by~\eqref{eq:16} and~\eqref{eq:18})
\[ \vep+n+2^{1-d}\ge \ero(G)>\vep+n+1-(1-2^{-q_1})=\vep+n+2^{-q_1}, \]
we obtain $q_1\ge d$. We show that $G+K_{1,q_1}\in\mac{L}_{\le\vep+n+1}$. Indeed, if $H$ is a proper vertex-minor of $G+K_{1,q_1}$, then $H$ is the disjoint union of $H_1$ and $H_2$ where $H_1$ is a vertex-minor of $G$ and $H_2$ is a vertex-minor of $K_{1,q_1}$ such that at least one of these two containments is proper. If $H_1$ is a proper vertex-minor of $G$, then since $G\in\mac{L}_{\le\vep+n}$,
\[ \ero(H)=\ero(H_1)+\ero(H_2)\le \vep+n+1-2^{-q_1}<\vep+n+1, \]
and if $H_2$ is a proper vertex-minor of $K_{1,q_1}$, then
\[ \ero(H)=\ero(H_1)+\ero(H_2)\le \ero(G)+1-2^{1-q_1}\le\vep+n+1. \]
Thus $G+K_{1,q_1}\in\mac{L}_{\le\vep+n+1}$. This proves the first claim.

Now let $v$ be a vertex of $G$. By Proposition~\ref{prop:attachedstar} and the construction of $q_1$, for all $k\ge q_1$, 
\[ \ero(G+_vK_{1,k})\ge \ero(G)+1-2^{-k}>\vep+n+1, \]
and for all $1\le k< q_1-1$, by~Proposition~\ref{prop:avgdecompose},
\[ \ero(G+_vK_{1,k})\le \ero(G)+1-2^{-1-k}\le \ero(G)+1-2^{1-q_1}\le\vep+n+1. \]
Because $G+_vK_{1,q_1-1}$ is a proper induced subgraph of $G+_vK_{1,q_1}$ and the average cut-rank is strictly monotone with respect to the induced subgraph relation by Theorem~\ref{thm:avgvertexminor}, there is a unique $q_2=q_2(G,v)\in\{q_1-1,q_1\}$ such that
\begin{align*}
\ero(G+_vK_{1,k})&>\vep+n+1\quad\text{ for all }k\ge q_2,\\
\ero(G+_vK_{1,k})&\le\vep+n+1\quad\text{ for all }1\le k<q_2.
\end{align*}
In the formation of $G':=G+_vK_{1,q_2}$, let $x$ be the central vertex of $K_{1,q_2}$ that is adjacent to $v$ and $S:=V(K_{1,q_2})$. We show that $G'\in \mac{L}_{\le\vep+n+1}$. Indeed, suppose for the contrary that $H$ is an elementary vertex-minor of $G'$ with $V(G')=V(H)\cup\{u\}$ such that $\ero(H)>\vep+n+1$. By Proposition~\ref{prop:elementary}, $H$ is locally equivalent to one of $G'-u,(G'*u)-u$, and $(G'\wedge uw)-u$ for any $w$ adjacent to $u$ in $G'$. We may assume without loss of generality that $H$ is one of these graphs. There are three cases to consider.
\begin{enumerate}
\item If $H=G'-u$, then $u$ belongs to one of $V(G)\setminus\{v\}$, $\{v\}$, $\{x\}$, and $S\setminus\{x\}$.
\begin{enumerate}
\item If $u\in V(G)\setminus\{v\}$ then $H=(G-u)+_vK_{1,q_2}$. Because $\ero(G-u)\le \vep+n$, we have, by Proposition~\ref{prop:avgdecompose},
\[ \vep+n+1<\ero(H)\le\ero(G-u)+\ero(K_{1,q_2+1})< \vep+n+1, \]
a contradiction.
\item If $u=v$ then $H=(G-v)+K_{1,q_2}$. Similarly we obtain a contradiction.
\item If $u=x$ then $H$ is the disjoint union of $G$ with $q_2$ isolated vertices, so $H$ and $G$ have the same average cut-rank which is smaller than $\vep+n+1$, a contradiction.
\item If $u\in S\setminus\{x\}$ then $H=G+_vK_{1,q_2-1}$ which has average cut-rank smaller than $\vep+n+1$ by the definition of $q_2$, a contradiction.
\end{enumerate}
\item If $H=(G'*u)-u$, then from the first case we may assume that $u$ is not a leaf in $G'$, hence $u\not\in S\setminus\{x\}$. There are three subcases to consider.
\begin{enumerate}
\item If $u\in V(G)\setminus\{v\}$ then $H=( (G*u)-u)+_vK_{1,q_2}$, which leads to a contradiction.
\item If $u=v$ then $H-S$ is an elementary vertex-minor of $G$, $N_H(x)=(S\setminus\{x\})\cup N_G(v)$, and $H[S]$ is an attached star of $H$ of size $q_2+1$ with the central vertex $x$, so by Proposition~\ref{prop:avgdecompose},
\[ \vep+n+1<\ero(H)\le\ero(H-S)+\ero(K_{1,q_2+d_G(v)})<\vep+n+1, \]
a contradiction.
\item If $u=x$ then $H*z$ is isomorphic to $G+_vK_{1,q_2-1}$ where $z$ is a vertex in $S\setminus\{x\}$, thus has average cut-rank smaller than $\vep+n+1$, a contradiction.
\end{enumerate}
\item If $H=(G'\wedge uw)-u$, then we may assume that $u$ is neither a leaf nor a neighbor of a leaf in $G'$, because otherwise $w$ either is the unique neighbor of $u$ in $G'$ or can be chosen to be a leaf adjacent to $u$, and so $H=(G'\wedge uw)-u$ is isomorphic to $G'-w$, returning to the first case. 
There are two subcases to consider.
\begin{enumerate}
\item If $u\in V(G)\setminus\{v\}$ then we may assume that $w\ne v$ (if there is no other choice then $u$ is a leaf in $G'$).
Now $H-S$ is an elementary vertex-minor of $G$ and $H=(H-S)+_vK_{1,q_2}$. We obtain a contradiction.
\item If $u=v$ then because $G$ has no isolated vertices, we may choose $w=x$. Then $H[(V(G)\setminus\{v\})\cup \{x\}]$ is isomorphic to $G$ via some isomorphism bringing $x$ to $v$ and fixing every vertex in $V(G)\setminus\{v\}$. Furthermore, in $H$, $S\setminus\{x\}$ is an independent set and complete to $N_G(v)\cup\{x\}$ as well as anticomplete to $V(G)\setminus(N_G(v)\cup\{x\})$. Thus, for some $z\in S\setminus\{x\}$ we have $H*x*z$ is isomorphic to $G+_{v}K_{1,q_2-1}$, which brings a contradiction.
\end{enumerate}
\end{enumerate}
Therefore $G+_vK_{1,q_2}\in\mac{L}_{\le\vep+n+1}$, completing the proof of the lemma.
\end{proof}
Now we come to the construction. Let $\mac{F}_{\vep}:=\{K_{1,\lfloor1-\log(1-\vep)\rfloor}\}$ for all $\vep\in[0,1)$, and for all integers $k\ge0$,
\begin{align*}
\mac{F}_{\vep+2k+1}&:=\{F+K_{1,q_1(F)}:F\in\mac{F}_{ \vep+2k}\},\\
\mac{F}_{\vep+2k+2}&:=\{(F+K_{1,q_1(F)})+_vK_{1,q_2(F+K_{1,q_1(F)},v)}:F\in\mac{F}_{\vep+2k},v\in R(F)\},
\end{align*}
where $q_1(F)$ and $q_2(F,v)$ are defined as in Lemma~\ref{lemma:attach}.
Note that no graphs in $\mac{F}_{\vep+n}$ have isolated vertices.
\begin{corollary}\label{cor:minimal}$\mac{F}_{\vep+n}\subseteq \mac{L}_{\le\vep+n}$ for all $n\ge0$.
\end{corollary}
\begin{proof}By Lemma~\ref{lem:avgcliques}, $K_{1,\lfloor1-\log(1-\vep)\rfloor}\in\mac{L}_{\le\vep}$ for all $\vep\in[0,1)$. The conclusion thus follows inductively by Lemma~\ref{lemma:attach}.
\end{proof}

Here is another consequence of Lemma~\ref{lemma:attach}.
\begin{corollary}\label{cor:larger}For all $F\in\mac{F}_{\vep+2n}$ and $v\in R(F)$, $q_2(F+K_{1,q_1(F)},v)$ is at least $q_1(F)$, hence at least the maximum weight in $\pi(F)$.
	
\end{corollary}
\begin{proof}Let $H:=F+K_{1,q_1(F)}$. By Lemma~\ref{lemma:attach}, $q_1(F)$ is at least the maximum weight in $\pi(F)$, so $q_1(F)+1$ is the largest weight in $\pi(H)$, which implies that $q_1(H)\ge q_1(F)+1$. Also by Lemma~\ref{lemma:attach}, $q_2(H,v)\ge q_1(H)-1$, and thus $q_2(H,v)$ is at least $ q_1(F)$, hence at least the maximum weight in $\pi(F)$.
\end{proof}
Now we account for the restriction $v\in R(F)$ in the definition of $\mac{F}_{\vep+2n+2}$: Because $q_2(H,v)$ can possibly be equal to $q_1(F)$, to deduce Lemmas~\ref{lemma:diffweight}~and~\ref{lemma:isomorphic} we require that the copy of $K_{1,q_2(H,v)}$ attached to $v$ lies in a component different from a copy of $K_{1,q_1(F)}$.
\begin{lemma}\label{lemma:diffweight}
	For every $F\in\mac{F}_{\vep+n}$, $\pi(F)$ has exactly $n+1$ vertices, and in $\pi(F)$,
	no positive integer appears more than twice as a weight;
	if some weight appears twice then
	the corresponding vertices are in different components
	and one of them is the smallest weight in its component.
\end{lemma}
\begin{proof}We proceed by induction on $n$. When $n=0$ the lemma is trivial. Assuming that the lemma is true for $n=2k$, we shall show that it is also true for $n=2k+1$ and $2k+2$. Let $F\in\mac{F}_{\vep+2k}$ and consider $H:= F+K_{1,q_1(F)}\in\mac{F}_{\vep+2k+1}$. Set $S:=V(K_{1,q_1(F)})$. By Lemma~\ref{lemma:attach}, $q_1(F)$ is at least the maximum weight in $\pi(F)$, so the conclusion holds for $\pi(H)$ because it also holds for $\pi(F)$, which is done by the induction hypothesis.
	
	Now consider $G:=H+_vK_{1,q_2(H,v)}\in\mac{F}_{\vep+2k+2}$ for $v\in R(F)$. By Corollary~\ref{cor:larger}, $q_2(H,v)$ is at least $q_1(F)$ as well as the maximum weight in $\pi(F)$. So, since $v\in R(F)$, the weights in $\pi(F)$ are preserved in $\pi(G)$, hence by the induction hypothesis the conclusion for $\pi(G-S)$ indeed holds. Thus, to verify the conclusion for $\pi(G)$, it is enough to check two (unique) copies of $K_{1,q_2(H,v)}$ and $K_{1,q_1(F)}$ in $G$. But this is easy, since if $q_2(H,v)>q_1(F)$ then we are done, and if $q_2(H,v)=q_1(F)$ then those two copies must be in different components because $v\in R(F)$.
\end{proof}
\begin{lemma}\label{lemma:isomorphic}For every $\vep\in[0,1)$ and $n\ge0$, no two distinct forests in $\mac{F}_{\vep+n}$ are isomorphic.
\end{lemma}
\begin{proof}When $n=0$ the lemma holds trivially. Assume that the lemma holds for $n=2k$, we show that it also holds for $n=2k+1$ and $n=2k+2$. Consider $H_j:=F_j+K_{1,q_1(F_j)}\in\mac{F}_{\vep+2k+1}$ where $F_j\in \mac{F}_{\vep+2k}$ for $j=1,2$ and suppose that $H_1$ and $H_2$ are isomorphic. Since $q_1(F_j)$ is at least the maximum weight in $\pi(H_j)$ and $|\pi(H_j)|=2k+2$ for $j=1,2$, necessarily $q_1(F_1)=q_1(F_2)$ and so $F_1$ must be isomorphic to $F_2$, implying $H_1$ and $H_2$ are isomorphic.

Now consider $G_j:=H_j+_{v_j}K_{1,q_2(H_j,v_j)}\in\mac{F}_{\vep+2k+2}$ for $v_j\in R(F_j)$ for $j=1,2$. Assume that $G_1$ and $G_2$ are isomorphic, then by Lemma~\ref{lem:idenforests} $\pi(G_1)\cong\pi(G_2)$. 
For $j=1,2$, let $T_j$ be the component in $G_j$ containing the attached star $K_{1,q_2(H_j,v_j)}$ so that $T_j$ is not the component isomorphic to $K_{1,q_1(F_j)}$ in $G_j$, by construction. By Corollary~\ref{cor:larger}, $q_2(H_j,v_j)$ is at least the maximum weight in $\pi(F_j)$, so by Lemma~\ref{lemma:diffweight} $q_2(H_j,v_j)$ is at least the maximum weight in $\pi(T_j)$, for $j=1,2$. Thus, necessarily $q_2(H_1,v_1)=q_2(H_2,v_2)$ and $\pi(T_1)\cong\pi(T_2)$, which leads to $\pi(G_1-V(T_1))\cong\pi(G_2-V(T_2))$. Hence, by deleting the vertex with label $q_2(H_1,v_1)=q_2(H_2,v_2)$ in each $\pi(G_j)$, we obtain $\pi(H_1)\cong\pi(H_2)$, so by Lemma~\ref{lem:idenforests} there is an isomorphism $\varphi$ from $H_1$ to $H_2$. Thus, because the labels in $T_j$ are distinct for $j=1,2$ by Lemma~\ref{lemma:diffweight}, we have $\varphi(v_1)=v_2$. Therefore $G_1$ and $G_2$ are isomorphic and the proof is completed.
\end{proof}
Combining Lemmas~\ref{lemma:diffweight}~and~\ref{lemma:isomorphic}, we deduce the number of pairwise nonisomorphic graphs in $\mac{F}_{\vep+n}$ for all $\vep\in[0,1)$ and $n\ge0$. We employ the standard notation $k!!=\prod_{1\le j\le k,j\equiv k\pmod 2}j$ for $k=1,2,\ldots$ and the convention $(-1)!!=0!!=1$.
\begin{corollary}\label{cor:cardinality}For every $\vep\in[0,1)$ and $k\ge0$, the number of pairwise nonisomorphic graphs in $\mac{F}_{\vep+2k}$ and $\mac{F}_{\vep+2k+1}$ is
\[|\mac{F}_{\vep+2k}|=|\mac{F}_{\vep+2k+1}|=(2k-1)!!.\]
\end{corollary}

The next lemma describes properties of $\mac{L}_{\le\alpha}$ and $\mac{L}_{<\alpha}$ to be used later.
\begin{lemma}\label{lem:setminus}
	Let $\alpha>0$.
	Then the following statements hold.
	\begin{itemize}
		\item $\mac{L}_{<\alpha}\setminus\mac{L}_{\le \alpha}$
		is the class of all graphs without isolated vertices of average cut-rank exactly $\alpha$.
		\item If $G\in \mac{L}_{\le\alpha}\setminus\mac{L}_{<\alpha}$,
		then $G$ has a proper vertex-minor $H$ of average cut-rank exactly $\alpha$ in $\mac{L}_{<\alpha}$ such that $|G|-|H|\le 2$.
		If the equality holds then $H$ can be chosen so that
		$G$ is isomorphic to $H+K_2$.
	\end{itemize}
\end{lemma}
\begin{proof}
  Let $G\in\mac{L}_{<\alpha}\setminus\mac{L}_{\le \alpha}$. Then 
  $G$ has no isolated vertices and
  $\ero(G)\ge \alpha$, so if $\ero(G)>\alpha$, $G$ must have a proper vertex-minor, say $H$, in $\mac{L}_{\le \alpha}$, but then $\ero(H)>\alpha$ so $G\not\in\mac{L}_{<\alpha}$ by definition, a contradiction. On the other hand, by Theorem~\ref{thm:avgvertexminor}, if a graph $G$ with no isolated vertices has average cut-rank $\alpha$ then $G\in\mac{L}_{<\alpha}\setminus\mac{L}_{\le\alpha}$.

Now let $G\in\mac{L}_{\le\alpha}\setminus\mac{L}_{<\alpha}$.
Then $G$ has a proper vertex-minor of average cut-rank at least~$\alpha$, say $H$, which also must have average cut-rank at most $\alpha$.
Thus $\ero(H)=\alpha$, and we may assume that $H\in\mac{L}_{<\alpha}$ by deleting isolated vertices.
Since $H$ is a proper vertex-minor of $G$, there is some $G'\in\mac{L}_{\le\alpha}$ locally equivalent to $G$ so that $H$ is a proper induced subgraph of $G'$.
Let $V(G)\setminus V(H)=\{v_0,\ldots,v_k\}$ where $k\ge0$ and $H':=G'-v_0$. We may assume that $k\ge 1$, because otherwise the lemma holds.
Because $H'$ is a proper vertex-minor of $G$ and contains $H\in\mac{L}_{<\alpha}$ as an induced subgraph, we have $\ero(H')=\ero(H)=\alpha$.
Then by Theorem~\ref{thm:avgvertexminor}, $\{v_1,\ldots,v_k\}=V(H')\setminus V(H)$ consists of isolated vertices in $H'$.
Hence, since $G'\in\mac{L}_{\le\alpha}$ has no isolated vertices, $G'[\{v_0,\ldots,v_k\}]$ is an attached star in $G'$ of size $k+1$ where $v_0$ is the central vertex.

If $v_0$ is isolated in $G'-v_1$, then $k=1$ and $G'[\{v_0,v_1\}]$ is a component of size $2$ in $G'$ and $G'$ is isomorphic to $H+K_2$.
Then $G$ is isomorphic to $H''+K_2$ where $H''$ is locally equivalent to $H$. 

If $v_0$ is not isolated in $G'-v_1$, then $G'-v_1$ has no isolated vertices and contains $H$ as a proper induced subgraph.
This implies, again by Theorem~\ref{thm:avgvertexminor}, that $\alpha=\ero(H)<\ero(G'-v_1)$, contradicting the minimality of $G$.

\end{proof}
We remark that if $\alpha$ is a positive integer, both $\mac{L}_{\le\alpha}\setminus\mac{L}_{<\alpha}$ and $\mac{L}_{<\alpha}\setminus\mac{L}_{\le\alpha}$ are nonempty. For instance,  $(2\alpha-1)K_2+K_{1,2}$ belongs to $\mac{L}_{\le\alpha}\setminus\mac{L}_{<\alpha}$ and $2\alpha K_2$ belongs to $\mac{L}_{<\alpha}\setminus\mac{L}_{\le\alpha}$.

To finish the proof of Theorem~\ref{thm:main4} we need one more lemma.
\begin{lemma}\label{lemma:vepplusn}If $\vep>0$ or $n\ge1$, then every forest $F$ in $\mac{F}_{\vep+n}\setminus\mac{L}_{<\vep+n}$ has a leaf, say $v$, whose deletion yields a forest, say $H$, in $\mac{L}_{<\vep+n}$ of average cut-rank exactly $\vep+n$. Moreover, if $v$ belongs to an equivalence class of size $2$ of $(V(F),\equiv_F)$ and its unique neighbor has degree $2$ in $F$ then $|\pi(H)|=n$; otherwise $|\pi(H)|=n+1$.
\end{lemma}
\begin{proof}Let $F\in\mac{F}_{\vep+n}\setminus\mac{L}_{<\vep+n}$. By Corollary~\ref{cor:minimal} and Lemma~\ref{lem:setminus}, $F$ has a proper vertex-minor, say~$H'$, of average cut-rank $\vep+n$ such that $H'\in\mac{L}_{<\vep+n}$ and $|F|-|H'|\le2$. Moreover, if $|F|-|H'|=2$ then $H'$ can be chosen so that $F$ is isomorphic to $H'+K_2$, so $F$ has a component of size $2$. In this case, by the construction of $\mac{F}_{\vep+n}$, Lemma~\ref{lemma:attach}, and Corollary~\ref{cor:larger} we deduce that $n\le1$. If $n=0$ then $F$ is isomorphic to $K_2$, so $H'$ is empty, but this is absurd since $\vep>0$ by hypothesis; if $n=1$ then $H'$ is isomorphic to $K_{1,q}$ for some $q\ge1$, a contradiction since $1\le1+\vep=\ero(H')=1-2^{-q}<1$. Thus, $H'$ is an elementary vertex-minor of $F$.

Let $\{x\}=V(F)\setminus V(H')$, then by Proposition~\ref{prop:elementary} we may assume without loss of generality that $H'$ is one of $F-x$, $(F*x)-x$, and $(F\wedge xy)-x$ for any $y$ adjacent to $x$ in $F$.
\begin{enumerate}
\item If $H'=F-x$ then since every equivalence class of $(V(F),\equiv_F)$ has at least two vertices (the construction of $\mac{F}_{\vep}$, Lemma~\ref{lemma:attach}, and Corollary~\ref{cor:larger}) and $H'$ has no isolated vertices, $x$ is necessarily a leaf in~$F$, so we let $v=x$.
\item If $H'=(F*x)-x$ then we may assume that $d_F(x)\ge2$, so if $y$ is a leaf adjacent to $x$ then $F-y$ is isomorphic to $H'*y\in\mac{L}_{<\vep+n}$, and we let $v=y$.
\item If $H'=(F\wedge xy)-x$ then if furthermore $x$ is a leaf in $F$ then $y$ is the unique neighbor of $x$ in~$F$, hence isolated in $H'$, a contradiction. So, $d_F(x)\ge2$, and since $y$ can be chosen to be any neighbor of $x$ in $F$, we may assume that $y$ is a leaf adjacent to $x$. Then $F-y$ is isomorphic to $ H'\in\mac{L}_{<\vep+n}$ and we let $v=y$.
\end{enumerate}

So, we have chosen $v$. Let $H:=F-v$ and $u$ be the unique neighbor of $v$ in $F$. 
In all cases, $H$ is locally equivalent to $H'$ and therefore $H\in\mac{L}_{<\vep+n}$.
The first part of the lemma is proved.

We come to the second part of the lemma. If $v$ belongs to an equivalence class of size $2$ of $(V(F),\equiv_F)$ and $d_F(u)=2$ then the neighbor of $u$ other than $v$ in $F$, say $w$, has degree at least two in $F$. Let $C$ be the equivalence class of $(V(F),\equiv_F)$ containing $w$, then $C\cup\{u\}$ is an equivalence class of $(V(H),\equiv_H)$. It follows that $|\pi(H)|=|\pi(F)|-1=n$ by Lemma~\ref{lemma:diffweight}.

In the other cases, it is easy to check that $|\pi(H)|=|\pi(F)|=n+1$. This completes the proof of the lemma.
\end{proof}

We are now ready to prove Theorem~\ref{thm:main4}.
\boundforforbiddenvertexminor*
\begin{proof}%
  Choose $c>0$ to be some constant (independent of $\vep$ and $n$) such that
\[ \frac{(2\lfloor n/2\rfloor-1)!!}{n+1}\ge 2^{cn\log (n+1)}\quad\text{for all $n\in\mab{N}$.} \]

First consider the case that $\mac{S}$ is a list of forbidden vertex-minors for 
the class of graphs of average cut-rank at most $\vep+n$. Then $\mac{S}$ is locally equivalent to $\mac{L}_{\le\vep+n}$.
By Corollary~\ref{cor:minimal}, $\mac{F}_{\vep+n}\subseteq\mac{L}_{\le\vep+n}$, and by Lemmas~\ref{lemma:forests}~and~\ref{lemma:isomorphic}, no two distinct forests $F_1$ and $F_2$ in $\mac{F}_{\vep+n}$ 
are locally equivalent up to isomorphisms.
Therefore, for every forest $F$ in $\mac{F}_{\vep+n}$, there is some member in $\mac{S}$ which is isomorphic to a graph locally equivalent to $F$ and these members are pairwise not locally equivalent to each other. By Corollary~\ref{cor:cardinality},
\begin{equation}\label{eq:19}
|\mac{S}|\ge |\mac{F}_{\vep+n}|=\left(2\lr\frac{n}{2}\rr-1\right)!!\ge 2^{cn\log (n+1)}.
\end{equation}

Now consider the case that $\mac{S}$ is a list of forbidden vertex-minors for the class of graphs of average cut-rank smaller than $\vep+n$. 
Then $\mac{S}$ is locally equivalent to $\mac{L}_{<\vep+n}$.
We may assume that $\vep+n>0$. Let $\{F_1,\ldots,F_m\}=\mac{F}_{\vep+n}\setminus\mac{L}_{<\vep+n}$.

For every $j=1,\ldots,m$, by Lemma~\ref{lemma:vepplusn}, $F_j$ has a leaf, whose deletion yields a forest in $\mac{L}_{<\vep+n}$, say~$H_j$, of average cut-rank exactly $\vep+n$. Moreover, $|\pi(H_j)|$ is either $n$ or $n+1$ depending on the condition written in the statement of Lemma~\ref{lemma:vepplusn}.
\begin{claim}For every $j\in\{1,\ldots,m\}$, there are, up to isomorphism, at most $n+1$ forests $F$ such that there is some leaf in $F$ whose deletion yields $H_j$.
\end{claim}
\begin{subproof}There are two cases to consider.
\begin{enumerate}
\item $|\pi(H_j)|=n$. The only way to obtain $F$ from $H_j$ is to add a new vertex to $H_j$ and join it to some leaf in $H_j$ (to create a new equivalence class of size $2$). Because $(V(H_j),\equiv_{H_j})$ has $n$ equivalence classes, each of which induces an attached star in $H_j$, there are at most $n$ forests $F$ satisfying the claim.
\item $|\pi(H_j)|=n+1$. The only way to obtain $F$ from $H_j$ is to add a new vertex to $H_j$ and join it to the central vertex of some equivalence class of $(V(H_j),\equiv_{H_j})$. Because there are $n+1$ such equivalence classes, there are thus at most $n+1$ forests $F$ satisfying the claim.
\end{enumerate}
Hence there are at most $n+1$ desired forests $F$, completing the proof of the claim.
\end{subproof}

Let $\mac{G}$ be a graph on the vertex set $\{1,\ldots,m\}$ such that for distinct $j,k\in\{1,\ldots,m\}$, $jk\in E(\mac{G})$ if $H_{j}$ is isomorphic to a graph locally equivalent to $H_{k}$. For $j\in\{1,\ldots,m\}$, by Lemma~\ref{lemma:forests}, $k\in N_{\mac{G}}(j)$ if and only if $H_j$ is isomorphic to $ H_k$, implying that there is some forest $F_k'$ isomorphic to $F_k$ such that $H_j$ can be obtained by deleting some leaf of $F_k'$. Because the set $\{F_1,\ldots,F_m\}$ consists of pairwise nonisomorphic forests, by Lemma~\ref{lemma:isomorphic}, so does the set $\{F_k':k\in N_{\mac{G}}(j)\}\cup\{F_j\}$. It follows by the claim that $d_{\mac{G}}(j)\le n$ for all $j=1,\ldots,m$.

Let $S$ be a maximal independent set in~$\mac{G}$. Then every vertex outside of $S$ is adjacent in $\mac{G}$ to some vertex in $S$ whose degree is at most~$n$. Hence $m=|\mac{G}|\le |S|+n|S|$, or equivalently $|S|\ge \frac{m}{n+1}$.

Let $\mac{T}$ be the disjoint union of $\mac{F}_{\vep+n}\cap\mac{L}_{<\vep+n}$ and $\{H_j:j\in S\}$. Since $S$ is an independent set in $\mac{G}$, for every distinct $j,k\in S$ we have $H_j$ is not isomorphic to a graph locally equivalent to $H_k$. This implies, from our construction, that $\mac{T}\subseteq\mac{L}_{<\vep+n}$ and no two distinct graphs in $\mac{T}$ are locally equivalent to each other up to isomorphisms. Furthermore, no two distinct forests in $\mac{F}_{\vep+n}\cap \mac{L}_{<\vep+n}$ are locally equivalent to each other up to isomorphisms. Therefore, by \eqref{eq:19},
\begin{equation*}
|\mac{S}|\ge|\mac{T}|=|\mac{F}_{\vep+n}\cap\mac{L}_{<\vep+n}|+|S|\ge |\mac{F}_{\vep+n}|-m+\frac{m}{n+1}\ge \frac{|\mac{F}_{\vep+n}|}{n+1}=\frac{(2\lfloor n/2\rfloor-1)!!}{n+1}\ge 2^{cn\log (n+1)},
\end{equation*}
and the theorem is completely proved.
\end{proof}

\section{Graphs of average cut-rank at most $3/2$}
\label{sec:3/2}
We now aim to prove Theorem~\ref{thm:3/2}.
Our plan is to bound the number of connected components and investigate the maximum induced path of every graph locally equivalent to a fixed graph. This approach not only characterizes graphs of average cut-rank at most $3/2$ but also reveals $\mac{L}_{<1},\mac{L}_{\le 1},\mac{L}_{<3/2},\mac{L}_{\le 3/2}$ up to local equivalence.
For every graph $G$, we denote by $p(G)$ the maximum length of a path graph which is a vertex-minor of $G$.

Recall that for every $k\ge0$, $E_k$ is the graph $K_{1,k+1}$ with one edge subdivided. The following lemma computes $\ero(E_k)$, which explains why $3/2-3/2^{k+2}$ appears in Theorem~\ref{thm:3/2}.
We omit its easy proof.
\begin{lemma}\label{lem:ek}For all $k\ge0$, we have
\( \ero(E_k)=\frac{3}{2}-\frac{3}{2^{k+2}}\).
\end{lemma}
\subsection{Graphs of average cut-rank at most $1$}
We need the following lemma. We leave its easy proof to the readers.
\begin{lemma}\label{lem:cograph}If $G$ is a connected graph having no path of length three as a vertex-minor then $G$ is isomorphic to a star or a complete graph.
\end{lemma}
\begin{lemma}\label{lemma:1}
  If a graph with no isolated vertices has average cut-rank at most~$1$, then it is isomorphic to a graph locally equivalent to one of $2K_2$ and $K_{1,k}$ for $k\ge1 $. Moreover
$\mac{L}_{<1}$ is locally equivalent to $\{2K_2,K_4\}$ and \(\mac{L}_{\le1}\) is locally equivalent to \(\{3K_2,K_2+P_3,P_4\} \).
\end{lemma}
\begin{proof}
  It follows easily from Lemma~\ref{lem:cograph} and the following observations: $\ero(2K_2)=1$, $\ero(3K_2)=3/2$, $\ero(K_2+P_3)=7/4$, and $\ero(P_4)=9/8$.
\end{proof}

\subsection{Graphs of average cut-rank at most $3/2$}
Let us start with several technical results whose proofs are left to the interested readers.

\begin{lemma}\label{lemma:five}Let $P$ be an induced path of length $3$ in a graph $G$ and $v$ be a vertex of $G$ outside~$P$ such that $v$ has at least $2$ neighbors in $P$. Then
\begin{itemize}
\item If $v$ is adjacent to both ends of $P$, $G$ contains a cycle of length $5$ as a vertex-minor.
\item Otherwise, $G$ contains a path of length $4$ as a vertex-minor.
\end{itemize}
\end{lemma}
\begin{lemma}\label{lem:lcfive}Every graph without isolated vertices on $5$ vertices is isomorphic to a graph locally equivalent to to one of $K_2+P_3$, $K_{1,4}$, $P_5$, $E_2$, and $C_5$.
\end{lemma}
\begin{lemma}\label{lem:fourandfive}
  Let $G$ be a graph on at most $5$ vertices. If $\ero(G)\ge3/2$, then $G$ is isomorphic to a graph locally equivalent to $C_5$.
\end{lemma}

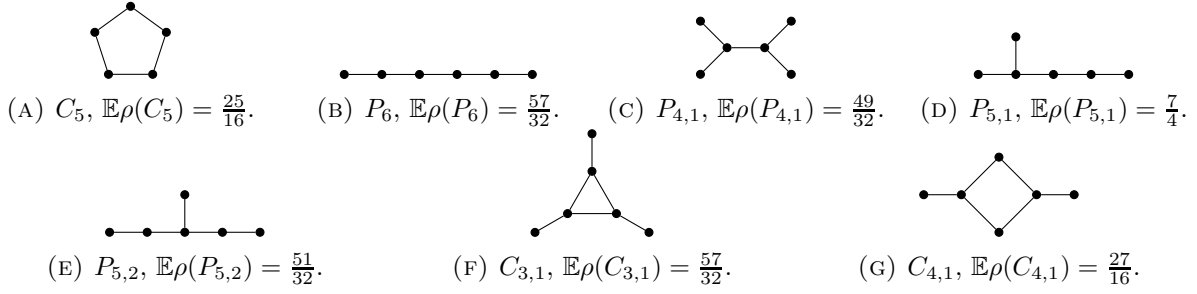
\begin{figure}
  \tikzstyle{v}=[circle, draw, solid, fill=black, inner sep=0pt, minimum width=3pt]
\begin{subfigure}[b]{.24\textwidth}
  \centering
  \begin{tikzpicture}[scale=.5]
    \foreach \j in {0,1,...,4}
    {\node [v] (v\j) at (360/5*\j+18:1){};
    }
	\draw (v0)--(v1)--(v2)--(v3)--(v4)--(v0);
  \end{tikzpicture}
  \caption{$C_5$, $\ero(C_5)=\frac{25}{16}$.}
\end{subfigure}
\begin{subfigure}[b]{.24\textwidth}
  \centering
  \begin{tikzpicture}[scale=.5]
    \foreach \j in {0,1,...,5}
    {\node [v] (v\j) at (\j,0){};
    }
	\draw (v0)--(v5);
  \end{tikzpicture}
  \caption{$P_{6}$, $\ero(P_6)=\frac{57}{32}$.}
\end{subfigure}
\begin{subfigure}[b]{.24\textwidth}
  \centering
  \begin{tikzpicture}[scale=.5]
    \foreach \j in {0,1}
    {\node [v] (v\j) at (\j+1,0){};
    }
	\draw (v0)--(v1);
    \draw (v1)--+(45:1) node[v]{};
	\draw (v1)--+(-45:1) node[v]{};
	\draw (v0)--+(135:1) node[v]{};
	\draw (v0)--+(-135:1) node[v]{};
  \end{tikzpicture}
  \caption{$P_{4,1}$, $\ero(P_{4,1})=\frac{49}{32}$.}
\end{subfigure}
\begin{subfigure}[b]{.24\textwidth}
  \centering
  \begin{tikzpicture}[scale=.5]
    \foreach \j in {0,1,...,4}
    {\node [v] (v\j) at (\j,0){};
    }
	\draw (v0)--(v4);
	\draw (v1)--+(90:1) node[v]{};
  \end{tikzpicture}
  \caption{$P_{5,1}$, $\ero(P_{5,1})=\frac{7}{4}$.}
\end{subfigure}
\begin{subfigure}[b]{.32\textwidth}
  \centering
  \begin{tikzpicture}[scale=.5]
    \foreach \j in {0,1,...,4}
    {\node [v] (v\j) at (\j,0){};
    }
	\draw (v0)--(v4);
	\draw (v2)--+(90:1) node[v]{};
  \end{tikzpicture}
  \caption{$P_{5,2}$, $\ero(P_{5,2})=\frac{51}{32}$.}
\end{subfigure}
\begin{subfigure}[b]{.32\textwidth}
  \centering
  \begin{tikzpicture}[scale=.5]
    \foreach \j in {0,1,2}
    {\node [v] (v\j) at (120*\j+90:0.75){};
    }
	\draw (v0)--(v1)--(v2)--(v0);
	\draw (v0)--+(90:1) node[v]{};
	\draw (v1)--+(-150:1) node[v]{};
	\draw (v2)--+(-30:1) node[v]{};
  \end{tikzpicture}
  \caption{$C_{3,1}$, $\ero(C_{3,1})=\frac{57}{32}$.}
\end{subfigure}
\begin{subfigure}[b]{.32\textwidth}
  \centering
  \begin{tikzpicture}[scale=.5]
    \foreach \j in {0,1,2,3}
    {\node [v] (v\j) at (90*\j:1){};
    }
	\draw (v0)--(v1)--(v2)--(v3)--(v0);
	\draw (v0)--+(0:1) node[v]{};
	\draw (v2)--+(-180:1) node[v]{};
  \end{tikzpicture}
  \caption{$C_{4,1}$, $\ero(C_{4,1})=\frac{27}{16}$.}
\end{subfigure}
\caption{Small graphs and their average cut-rank.}\label{fig:smallgraphs}
\end{figure}
Graphs $C_5$, $P_6$, $P_{4,1}$, $P_{5,1}$, $P_{5,2}$, $C_{3,1}$, and $C_{4,1}$ with their average cut-rank are listed in Figure~\ref{fig:smallgraphs}.
We deduce the following easily.
\begin{corollary}\label{cor:l3/2}The graphs $C_5$, $P_6$, $P_{4,1}$, $P_{5,1}$, $P_{5,2}$, $C_{3,1}$, and $C_{4,1}$ belong to $\mac{L}_{<3/2}\cup\mac{L}_{\le3/2}$.
\end{corollary}
The following lemma is a major step toward the proof of Theorem~\ref{thm:3/2}.

\begin{lemma}\label{lemma:3/2}
  If a graph without isolated vertices has average cut-rank at most $3/2$, then 
  it is isomorphic to a graph locally equivalent to one of $P_5$, $3K_2$, $2P_3$, $K_{1,k+1}$, $K_2+K_{1,k+1}$, and $E_k$ for $k\ge0$. Moreover
\begin{align*}
\mac{L}_{<3/2}&\simeq\{3K_2,K_2+P_3,2P_3,C_5,P_{6},P_{4,1},P_{5,1},P_{5,2},C_{3,1},C_{4,1}\},\\
\mac{L}_{\le 3/2}&\simeq\{4K_2,2K_2+P_3,K_2+P_4,P_3+K_{1,3},P_3+P_4,C_5,P_{6},P_{4,1},P_{5,1},P_{5,2},C_{3,1},C_{4,1}\}.
\end{align*}
\end{lemma}
\begin{proof}Let $G$ be a graph such that either $\ero(G)\le3/2$ or $G\in\mac{L}_{<3/2}\cup\mac{L}_{\le3/2}$.
  By Lemmas~\ref{lem:lcfive} and \ref{lem:fourandfive}, we may assume that $\abs{G}>5$.
  It is easy to check that $C_5, P_6, 4K_2, 2K_2+P_3\in \mac{L}_{\le3/2}$ and so we may assume that $G$ has no vertex-minor isomorphic to $C_5$, $P_6$, $4K_2$, $2K_2+P_3$. Thus $G$ has at most $3$ components.

  If $G$ has exactly $3$ components, then $G$ has an induced subgraph isomorphic to $3K_2$ which has average cut-rank $3/2$, and if furthermore $G$ has at least $7$ vertices then it has a vertex-minor isomorphic to $2K_2+P_3$ whose average cut-rank is $7/4$. Hence if $\ero(G)\le3/2$ then $G$ is isomorphic to $3K_2$, if $G\in\mac{L}_{<3/2}$ then $G$ is isomorphic to $3K_2$, and if $G\in\mac{L}_{\le3/2}$ then $G$ is isomorphic to a graph locally equivalent to $2K_2+P_3$.

When $G$ has exactly $2$ components, if every component of $G$ has at least $3$ vertices then $G$ has a vertex-minor isomorphic to $2P_3$ whose average cut-rank is $3/2$, and if furthermore $G$ has at least $7$ vertices then $G$ has a vertex-minor isomorphic to $P_3+K_{1,3}$ or $P_3+P_4$ whose average cut-rank is $13/8$ or $15/8$, respectively; if one component of $G$ has only $2$ vertices then $G=K_2+H$ for some graph $H$ with $\ero(H)=\ero(G)-1/2$. By applying Lemma~\ref{lemma:1} to $H$, we deduce that if $\ero(G)\le3/2$ then $G$ is isomorphic to a graph locally equivalent to $K_2+K_{1,k}$ for some $k\ge1$ or $2P_3$, if $G\in\mac{L}_{<3/2}$ then $2P_3$, and if $G\in\mac{L}_{\le3/2}$ then $G$ is isomorphic to a graph locally equivalent to $P_3+K_{1,3}$ or $P_3+P_4$.

Now we assume that $G$ is connected. By Lemma~\ref{lemma:1}, we may assume that $G$ has average cut-rank larger than $1$, so by Lemma~\ref{lem:cograph}, $p(G)\ge3$.
By applying local complementaions if necessary,
we may assume that $G$ has an induced path of length $p(G)$.

If $p(G)\ge 4$, then let $P=abcde$ be an induced path   of length $4$. Then there is a vertex $v$ outside $P$ adjacent to some vertex of $P$. 
If $v$ is adjacent to $a$, then it is easy to check that 
$G$ has a vertex-minor isomorphic to $C_5$ or $P_6$,  contradicting our assumption.
Thus $v$ is nonadjacent to $a$
and by symmetry, nonadjacent to $e$.
By considering all possible $N(v)\cap\{b,c,d\}$, 
we deduce that $G$ has a vertex-minor isomorphic to 
$P_{5,1}$, $P_{5,2}$, $C_{4,1}$, or $C_{3,1}$.
Hence, if $p(G)=4$, then $\ero(G)>3/2$ and in addition if $G\in\mac{L}_{<3/2}\cup\mac{L}_{\le3/2}$ then $G$ is isomorphic to a graph locally equivalent to one of $P_{5,1}$, $P_{5,2}$, $C_{4,1}$, and $C_{3,1}$, by Corollary~\ref{cor:l3/2}.

If $p(G)=3$ then let $P=abcd$ be an induced path of length $3$.
Then by Lemma~\ref{lem:lcfive}, $S:=N_G(P)\setminus V(P)\ne\emptyset$. Pick $v\in S$. 
If $\{a,d\}\subseteq N_G(v)$ then $G$ is isomorphic to a graph locally equivalent to $C_5$, contradicting our assumption. 
Thus we may assume that $v$ is nonadjacent to $d$. 
If $v$ is adjacent to $a$, then we may apply local complementations to find a vertex-minor isomorphic to $P_5$, contradicting the assumption that $p(G)=3$. 
Thus, $v$ is nonadjacent to $a$.
By the same argument, we deduce that $v$ is adjacent to exactly one of $b$ and $c$.
Hence, each vertex in $S$ should be adjacent to only one of $b$, $c$ in $P$.
If all the vertices in $S$ are pairwise nonadjacent and adjacent to the same among $b$, $c$,
then $G$ is isomorphic to $E_k$ for some $k\ge 1$ and thus $\ero(G)<3/2$.
Otherwise, there are two vertices in $S$,
say $u,v$, being adjacent to each other or adjacent to different vertices in $\{b,c\}$.
In the case $u,v$ are adjacent,
if they are adjacent to the same among $b,c$ then
$P_{5,1}$ is isomorphic to an induced subgraph of $G*u$ which contradicts $p(G)=3$,
otherwise $G\wedge uv$ has an induced path of length $5$,
contradicting the assumption; 
in the case $u,v$ are adjacent to different vertices in $\{b,c\}$,
we only have to check when $uv\nin E(G)$, 
then $G$ is isomorphic to a graph locally equivalent to $P_{4,1}$.
Thus, if $p(G)=3$ and $\ero(G)\le3/2$ then
$G$ is isomorphic to a graph locally equivalent to $E_k$ for some $k\ge0$, 
and if $p(G)=4$ and $G\in\mac{L}_{<3/2}\cup\mac{L}_{\le3/2}$ 
then $G$ is isomorphic to a graph locally equivalent to $P_{4,1}$ by Corollary~\ref{cor:l3/2}.
\end{proof}
\smallgraphs*
\begin{proof}%
  It suffices to combine Proposition~\ref{prop:avgdecompose}, and Lemmas~\ref{lem:ek}, \ref{lem:avgcliques}, \ref{lemma:3/2}, and the fact that $\ero(P_5)=23/16=3/2-1/2^4$.
\end{proof}

\section*{Acknowledgments}
The authors would like to thank the anonymous reviewers for helpful comments.

\end{document}